\documentclass[11pt]{article}

\usepackage{amsmath,amsthm,amssymb,color,authblk}
\usepackage{epsfig,amsfonts,latexsym, hyperref, bbm, mathrsfs, longtable}
\usepackage{natbib}
\usepackage{geometry}
\numberwithin{equation}{section}

\newcommand{\h}{\hspace{1cm}}

\linethickness{.25mm}

\newcommand{\prob}{\mathbb{P}}
\newcommand{\var}{\mathbb{V}{\rm ar}}
\newcommand{\exptn}{\mathbb{E}}

\newcommand{\calf}{\mathcal{F}}
\newcommand{\dtv}{{\rm d}}

\newcommand{\ccr}{C_c^+(\bar{\mbbr}_0)}
\newcommand{\mbbr}{\mathbb{R}}
\newcommand{\mbbn}{\mathbb{N}}
\newcommand{\mbbo}{\mathbbm{1}}

\newcommand{\sfT}{\mathsf{T}}

\newcommand{\sfu}{\mathsf{u}}
\newcommand{\sfv}{\mathsf{v}}
\newcommand{\bdelta}{\boldsymbol{\delta}}
\newcommand{\mbbM}{\mathbb{M}}
\newcommand{\mbbt}{\mathbb{T}}
\newcommand{\scrt}{\mathscr{T}}
\newcommand{\scrm}{\mathscr{M}}

\newcommand{\scrl}{\mathscr{L}}
\newcommand{\dist}{{\rm dist}}
\newcommand{\wt}{\widetilde}

\newcommand{\ol}{\overline}

\newcommand{\wtbfN}{\wt{\bf N}}
\newcommand{\hlsfun}{F_{g_1,g_2, \epsilon_1, \epsilon_2}}

\DeclareMathOperator{\pow}{Pow}

\newcommand{\beq}{\begin{equation}}
\newcommand{\eeq}{\end{equation}}
\newcommand{\alns}[1]{\begin{align*}#1\end{align*}}
\newcommand{\aln}[1]{\begin{align} #1 \end{align}}
\newcommand{\been}{\begin{enumerate}}
\newcommand{\een}{\end{enumerate}}
\newcommand{\norm}[1]{\| #1 \|}
\newcommand{\inv}{{-1}}
\newcommand{\probconv}{\stackrel{\prob}{\longrightarrow}}

\newcommand{\hlconv}{\xrightarrow{\mbbM_0}}

\newcommand{\asconv}{\xrightarrow{a.s.}}
\newcommand{\eqd}{\stackrel{d}{=}}
\newcommand{\nn}{\nonumber}

\newcommand{\rtext}[1]{\textcolor{black}{#1}}

\usepackage{todonotes}

\usepackage{cleveref}
\newtheorem{thm}{Theorem}[section]
\newtheorem{propn}[thm]{Proposition}
\newtheorem{lemma}[thm]{Lemma}

\theoremstyle{remark}
\newtheorem{remark}[thm]{Remark}
\newtheorem{cor}[thm]{Corollary}

\theoremstyle{definition}
\newtheorem{defn}[thm]{Definition}
\newtheorem{fact}[thm]{Fact}
\newtheorem{ass}[thm]{Assumptions}

\allowdisplaybreaks

\title{{ Large deviations of extremes in branching random walk  with regularly varying displacements}}

\author{{Ayan} {Bhattacharya}}

\date{March 10, 2022}

\hypersetup{
    colorlinks=true,
    linkcolor=red,
    citecolor=blue,
    urlcolor=blue,
    pdfborder={0 0 0}
}

\begin{document}

\maketitle

\begin{abstract}
In this article, we consider a branching random walk on the real-line where displacements coming from the same parent have jointly regularly varying tails. The genealogical structure is assumed to be a supercritical Galton-Watson tree, satisfying Kesten-Stigum condition. We study the large deviations of the extremal process, formed by the appropriately normalized positions in the $n$-th generation and show that the large extreme-positions form clusters in the limit. As a consequence of this, we also study the large deviations of the maximum among positions at the $n$-th generation.
\end{abstract}

\section{Introduction} \label{sec:introduction}

Branching Random Walk (BRW) on the real-line is a generalization of the branching process where each particle is assigned a position on the real-line. A formal description is as follows: BRW starts with a single particle at the origin of the real-line. This particle is called the initial ancestor and its position is defined to be zero. After a unit time, the initial ancestor reproduces an independent copy of the point process $\scrl$ on the real-line and dies. The atoms of the reproduced copy of the point process are called displacements. The newborn particles form the first generation. Position of each of these particles is defined to be its displacement translated by the position of its parent. After a unit time, each of the particles in the first generation dies after producing an independent copy of the point process $\scrl$. The newborn particles form the second generation. We assign a position to each of them by translating their displacement by the position of their parent. This mechanism goes on that the resultant system is called BRW.

 We briefly describe now the set-up of BRW considered in this article. Let $\bdelta_x$ denote the Dirac measure which puts unit mass at $x$. \label{nota_dirac_delta} It will be assumed that the point process $\scrl$ has the form
\begin{align}
{\mathscr L} = \sum_{i=1}^Z \bdelta_{X_i}. \label{defn:progeny:pp:BRW}
\end{align}
It is clear that the genealogical structure underlying BRW is given by Galton-Watson (GW) tree with progeny distribution $Z$. Precise assumptions on genealogical structure are as follows.

\begin{ass}[Assumptions on the genealogical structure]  \label{ass_genealogy}
\begin{enumerate}
\item The  progeny random variable  $Z$ is allowed to be  $\mbbn_0 =\{0,1, 2, \ldots\}$-valued and independent of displacement random vector ${\bf X} = (X_i : i \ge 1)$.

\item The genealogical GW tree is assumed to be supercritical and satisfies Kesten-Stigum condition, that is,
\aln{
\mu := \exptn(Z) \in (1, \infty) ~~~~~~\mbox{ and } ~~~~~ \exptn(Z \log_+ Z) < \infty \label{eq:assumption:branching:process}
}
where $\log_+ x = \log(x \vee 1)$.
\end{enumerate}
\end{ass}

\begin{ass}\label{ass:marginal:identical:dist}
The displacements are assumed to be identically distributed and have regularly varying tail at $\infty$ of index $\alpha$.  To be precise,
\aln{
& \mbox{$X_i$'s are  identically distributed random variables such that } \nonumber \\
& \prob(|X_1|>x) = x^{-\alpha} {\rm L}(x) \mbox{ and } \lim_{x \to \infty} \frac{\prob(X_1 > x)}{\prob(|X_1| > x)} = q \in [0,1] \label{eq:marginal:regvar}
}
where ${\rm L}$ is a {\it slowly varying function} ($\lim_{t \to \infty} {\rm L}(tx)/ {\rm L}(t) = 1$ for every $x > 0$).
\end{ass}

  Later, the displacements ${\bf X}$ are assumed to be jointly regularly varying (see Assumption~\ref{ass:joint:regvar:disp}). If $Z \le b$ almost surely in \eqref{defn:progeny:pp:BRW}, then we can use multivariate regular variation to model the joint regular variation of the displacements coming from the same parent which is very restrictive. Thanks to \cite{lindskog:resnick:roy:2014} for their contributions to the regular variation on $\mbbr^\mbbn$ which will allow us to define joint regular variation for the random sequence ${\bf X}$. An important advantage of this set-up is that it covers asymptotic tail-independence, multivariate regular variation, asymptotically full tail-dependence (see subsection~\ref{subsec_linearly_dep_disp}) structure of the displacements.

\subsubsection{Scope of the article} \label{subsec_scope}
In this article, our aim is to study exact asymptotics (rate of decay) for the probabilities of occurrence of  {\it exceptionally large value of the extreme positions}. We introduce first some notations which will help to state the aim in a more formalized fashion. 
  The genealogical GW tree will be denoted by $\mathsf{T} =  ({\cal E}, {\sf V})$ \label{nota_tree_vertices_edges}  where ${\cal E}$ and ${\sf V}$ denote the set of all edges and vertices respectively. We use Ulam-Harris labeling for the vertices $\sfv \in {\sf V}$.   The position and the generation of the vertex $\sfv \in {\sf V}$ are denoted by $S(\sfv)$  \label{nota_position} and $|\sfv|$ \label{nota_generation} respectively. $X(\sfu)$ \label{nota_displacement} will be used to denote the displacement attached to the vertex $\sfu$. $Z_n$ \label{nota_Zn} denotes the size of the $n$-th generation of the tree ${\sf T}$ for every $n \ge 0$.  The $k$-th largest position at the $n$-th generation is denoted by $M_n^{(k)}$ \label{nota_order_stat_k} for $k \ge 1$. We are mainly interested in the large values of $M_n^{(1)} = \max_{|\sfv| = n}S(\sfv)$ and  $(M_n^{(i)} : 1 \le i \le k)$ through the point process.  We shall use $f_n = o(g_n)$ if $\lim_{n \to \infty} f_n/ g_n = 0$ and $f_n \sim g_n$ if $\lim_{n \to \infty} f_n/g_n = 1$  for two positive increasing sequences $(f_n : n \ge 1)$ and $(g_n : n \ge 1)$. Suppose that the \rtext{assumptions} \eqref{eq:assumption:branching:process} and \eqref{eq:marginal:regvar} hold. It is known in the literature (see  \cite{bhattacharya:hazra:roy:2017} and the references therein) that the \rtext{order of} fluctuation of the largest $k$ positions $(M_n^{(i)} : 1 \le i \le k)$ is given  by a sequence $(b_n : n \ge 1)$ such that 
\aln{
b_n = \inf \{ x : \prob ( |X_1| > x ) \le \mu^{- n} \} \sim \mu^{n/\alpha} \wt{\rm L}(\mu^n). \label{eq:defn:b_n}
}  
Here $\wt{\rm L}$ is a slowly varying function related to ${\rm L}$ in \eqref{eq:marginal:regvar}. By definition, $(b_n : n \ge 1)$ satisfies $\mu^{n} \prob(|X_1| > b_n) \to 1 $. Consider a sequence $(\gamma_n : n \ge 1)$ which grows faster than the sequence $(b_n : n \ge 1)$ such that 
\aln{
b_n = o(\gamma_n) \mbox{ that is, } \mu^n \prob(|X_1| > \gamma_n) = o(1). \label{eq:defn:gamman}
}
The main aim of this article is to study the exact asymptotics of the probability 
\aln{
\Big\{ \sum_{|\sfv| = n} \bdelta_{\gamma_n^\inv |S(\sfv)|} (y, \infty) \ge k \Big\}  \mbox{ for every } y > 0 \mbox{ and } k \ge 1 \label{eq_event_ldp_position}
}
and identify its limit conditioned on the survival of the genealogical tree. This task is accomplished using the tools developed in \cite{hult:samorodnitsky:2010} to study the large deviations of point processes. \rtext{ To be precise, we show that the probability of the events in \eqref{eq_event_ldp_position} can be normalized in such a way that the normalized probability converges to a (random) limit measure. The limit measure lives on the space of all point measures on $(0, \infty)$ equipped with the Borel sigma-algebra generated by the vague metric. We further show that the limit measure concentrates on the point measures with random clusters and the law of the random clusters can be described through the law of the genealogical GW tree.} The appearance of the clusters can be explained by the fact that {\it a large displacement can cause multiple large positions} which is very specific and natural for this model. In this article, we first uncover this picture rigorously. This investigation is used later for the exact asymptotics of the probability of  $ \{ M_n^{(1)} > \gamma_n y \}$. Recently, \cite{gantert:2018} and \cite{dyszewski:gantert:hofelsauer:2020} have studied the large deviations of the maximum position when displacements have exponentially decaying tail.



 The joint regular variation of the tails of the displacements (from the same parent) plays an important role in our analysis. We have added a small discussion in subsection~\ref{subsec_linearly_dep_disp} emphasizing its importance. We mentioned only linearly and fully dependent displacements for the interest of the space. The extremes for the dependent stationary sequence of random variables with regularly varying tail have been studied thoroughly over the last few decades (see \cite{resnick:roy:2014}, \cite{basarak:segers:2009} and the large deviations of extreme positions have also been addressed in the works \cite{fasen:roy:2016} and \cite{hult:samorodnitsky:2010}). These studies may help us further to study the large deviations of order statistics, gap statistics, and other complicated but important functions of extreme positions. Our main result (Theorem~\ref{thm:extreme:positions:linear:dependent:disp} in Section~\ref{sec:mainresults}) is the key to the aforementioned investigations.


\subsubsection{ Outline of the paper} 

In Subsections~\ref{subsec_ldp_max_iid} and \ref{subsec_ldp_pp_iid}, we assume the displacements to be i.i.d. and discuss the large deviations of the maximum and point processes associated with $(\gamma_n^\inv S(\sfv) : |\sfv| = n)$ respectively. We briefly review the literature in Subsection~\ref{subsec_review_literature} and mention some of the unsolved problems. Section~\ref{sec:mainresults} contains the main results of this article along with their consequences after a brief introduction to the $\mathbb{M}_0$-topology. A discussion on the proofs with some open questions are provided in Subsection~\ref{sesubsec:discussion:open:probs}. The rest of the paper is dedicated to the proofs of the main results and their consequences. A list of the notations is provided in Section~\ref{sec:notation}.

\subsection{Large deviations for the maximum position: i.i.d. displacements and immortal genealogical tree}
\label{subsec_ldp_max_iid}

Here, we assume that the genealogical tree does not have any leaf in addition to the assumptions \eqref{eq:assumption:branching:process}  for the sake of simplicity. It will be further assumed that the displacements are i.i.d. and have regularly varying tail (see assumption \eqref{eq:marginal:regvar}). Our first task is to guess the sequence $(r_n : n \ge 1)$ such that 
\aln{
\liminf_{n \to \infty} r_n \prob(M_n^{(1)} > \gamma_n ) > 0. \label{eq_maxima_iid_first_task}
}
The second task is to establish that $\lim_{n \to \infty} r_n \prob(M_n^{(1)} > \gamma_n)$ exists  and then characterize the limit of $r_n \prob( M_n^{(1)} > \gamma_n y)$ for every $y > 0$.  In the next paragraph, we present heuristics to guess the sequence $(r_n : n \ge 1)$ based on the `principle of single large displacement'.



\noindent{\bf  Principle of single large displacement.} We assume $\exptn(Z^2) < \infty$ to present the heuristics. According to the {\it principle of single large displacement}, $M_n^{(1)}$ exceeds the value $\gamma_n$ if and only if exactly one of the displacements exceeds the threshold $\gamma_n$.  Mathematically,  this means
 \aln{
\prob(M_n^{(1)} > \gamma_n) \sim \prob \Big( \bigcup_{|\sfu| \le n} \{ X(\sfu) > \gamma_n \}  \Big).  \label{eq:principle:single:big:jump}
}
We further note that the right hand side of \eqref{eq:principle:single:big:jump} is bounded above by
\aln{
\exptn \Big( \sum_{|\sfu| \le n} \prob \Big[ X(\sfu) > \gamma_n ~ \Big| \cup_{i =1}^n{\sf D}_i \Big] \Big)   \le  \frac{\mu^{n + 1}}{(\mu -1)} \prob(X_1 > \gamma_n)  \label{eq_huristic_upper_bound}
}
using union bound as the displacements are independent of the branching mechanism where ${\sf D}_n = \{\sfu : |\sfu| = n\}$ for every $n \ge 0$ \label{nota_Dn}.  
 We also know that 
\aln{
&  \prob \Big( \bigcup_{|\sfu| \le n} \{ X(\sfu) > \gamma_n \}  \Big) \nonumber \\
&\ge \exptn \Big[ \sum_{|\sfu| \le n} \prob(X(\sfu) > \gamma_n ~ | \cup_{i = 1}^n {\sf D}_i )   - \sum_{|\sfu| \le n} ~~ \sum_{\sfu': \sfu' \neq \sfu; ~ |\sfu'| \le n} \prob \Big( X(\sfu) > \gamma_n;~~ X(\sfu') > \gamma_n ~ \Big| ~\cup_{i = 1}^n {\sf D}_i \Big) \Big] \nonumber \\
& \ge  \frac{\mu^{n + 1} - \mu}{\mu -1} \prob( X_1 > \gamma_n)  - \frac{1}{2} \big [\prob( X_1 > \gamma_n) \big]^2 \exptn \Big[ ( \sum_{i =1}^n Z_i)^2 \Big] \label{eq:heuristic:lower:bound:first:ineq}
}
as the displacements are i.i.d. and independent of the branching mechanism. It is easy to see that $\exptn(\sum_{i = 1}^n Z_i)^2 = \sum_{i =1}^n \exptn(Z_i^2) + 2\sum_{i = 1}^n \sum_{j = i + 1}^n \exptn(Z_i Z_j) $. Using the martingale property of $(\mu^{-n} Z_n : n \ge 0)$ adapted to the filtration $(\calf_n = \sigma(Z_0, Z_1, \ldots, Z_n) : n \ge 0)$, for every $j \ge i+1$, we have 
\alns{
\exptn(Z_i Z_j) = \exptn[ \mu^{ j} Z_i \exptn( \mu^{- j} Z_j | \calf_i) ] = \mu^{j - i} \exptn(Z_i^2).
}
So an upper bound of $\exptn(Z_i^2)$ yields an upper bound to the second term in \eqref{eq:heuristic:lower:bound:first:ineq}. It can be easily derived (see Example~4.4.9 in \cite{durrett:2019} for example) that 
\alns{
\exptn(Z_i^2) = \mu^{2i} ( 1 + \var(Z) \sum_{k = 2}^{i + 1} \mu^{- k}) \le \mu^{2i}\Big( 1 +\frac{\mu \var(Z)}{\mu - 1} \Big).
}  
These observations together yield 
\aln{
\exptn \Big[ (\sum_{i = 1}^n Z_i)^2 \Big] \le \Big( 1 + \frac{\mu \var(Z)}{\mu - 1} \Big)[ (\mu^2 - 1)^\inv + 2 (\mu - 1)^2] \mu^{2n + 2} =: c_2 \mu^{2n}. \label{eq:upper_bound_second_moment_totalsizen}
}
Combining \eqref{eq:heuristic:lower:bound}, \eqref{eq:upper_bound_second_moment_totalsizen} and \eqref{eq:defn:gamman},  we have following lower bound for the probability on the right hand side of \eqref{eq:principle:single:big:jump}
\aln{
 \frac{\mu^{n + 1} - \mu}{\mu -1} \prob( X_1 > \gamma_n)  - {\rm c}_2 \big[ \mu^n \prob( X_1 > \gamma_n) \big]^2  \sim \frac{\mu^{n + 1}}{\mu -1} \prob( X_1 > \gamma_n) . \label{eq:heuristic:lower:bound}
}
So combining \eqref{eq:principle:single:big:jump} and \eqref{eq:heuristic:lower:bound} with \eqref{eq_huristic_upper_bound},  we can see that $ \prob(M_n^{(1)} > \gamma_n) \sim \frac{\mu}{\mu - 1}  \Big( \mu^n \prob(X_1 > \gamma_n) \Big)$.
The heuristics then suggest
\aln{
r_n = \Big( \mu^n \prob(|X_1| > \gamma_n) \Big)^{-1} \label{eq:defn:rn}
} 
when the left-tail of the displacement can decay utmost at the order of the right-tail (tail balancing condition \eqref{eq:marginal:regvar}). The following result shows that this heuristic holds even when $\exptn(Z_1^2) = \infty$, but the Kesten-Stigum condition holds.


\begin{thm} \label{thm:ldp:maxima:iid}
Suppose that  the displacements $(X_i : i \ge 1)$ are i.i.d. random variables and satisfy \eqref{eq:marginal:regvar}.  Let $Z$ satisfy  $\prob(Z \ge 1) = 1$ in addition to \eqref{eq:assumption:branching:process}. Then, for every $x > 0$,  we have 
\begin{align}
\lim_{n \to \infty} r_n \prob \Big( M_n^{(1)} > \gamma_n x \Big) =  q x^{-\alpha}\mu/(\mu -1).
\end{align}
\end{thm}

\subsection{ Large deviations for point processes: i.i.d. displacements and immortal genealogical tree}
\label{subsec_ldp_pp_iid}

The most interesting and important feature of the model BRW is the dependence structure among the positions and the fact that one displacement can affect multiple positions. The large deviations of the maximum position does not offer any knowledge about the joint behavior of the large value of the extreme positions. Point process is the most popular tool in the extreme value theory to capture the dependence structure among the extreme positions.  For displacements with exponentially decaying tail, it has been predicted in \cite{brunet:derrida:2011}  that the weak limit of the extremal processes have clusters due to strong dependence structure. The prediction has been formally proved in \cite{madaule:2011} (see \cite{subag:zeitouni:2014} and the references therein).  The clustering phenomenon for the weak limit has been formalized in \cite{bhattacharya:hazra:roy:2016} (displacements are assumed to be i.i.d.) and \cite{bhattacharya:hazra:roy:2017} (displacements coming from the same parent have jointly regularly varying tail).  To the best of our knowledge, there is no article in the literature till now addressing the large deviations of the point process. In this article, we shall show that the large value of the extreme positions also form the clusters. We now develop some notations to discuss the large deviations of the point processes in detail.

By large deviations of the point process, we mean the asymptotic study of the following point process
\aln{
{\bf N}_n := \sum_{|\sfv| = n} \bdelta_{\gamma_n^\inv S(\sfv)} \label{eq_scope_of_article}
}
as $n \to \infty$. Let $\scrm(\mbbr_0)$ denote the space of all point measures on $\mbbr_0 := \mbbr \setminus \{0\}$. \label{nota_zero_removed_real_line} It is immediate that the limit of ${\bf N}_n$ (if exists) would put infinite mass at the point $0$ and so, we have ${\bf N}_n \probconv \varnothing$ where $\varnothing$ is the null measure in $\scrm(\mbbr_0)$ equipped with vague topology. If we consider a measurable subset ${\sf A} \subset \scrm_0 := \scrm(\mbbr_0) \setminus \{\varnothing\}$ which does not contain $\varnothing$ as its limit point (such sets will be called `nice' sets in this subsection), then $\prob({\bf N}_n \in {\sf A}) \to 0$ as $n \to \infty$. According to \cite{hult:samorodnitsky:2010}, this is exactly the framework where the large deviations of the point processes come into the  play. Motivated by this, we borrow their tools for the asymptotic analysis of ${\bf N}_n$.

Here we have two tasks to accomplish. The first task will to be guess the rate of the decay of the probabilities $\prob({\bf N}_n \in {\sf A})$ as $n  \to \infty$. Let ${\rm int }({\sf A})$ and ${\rm cl}({\sf A})$ denote interior and closure of the set ${\sf A}$ respectively.  Let ${\sf E}_x := \{ \xi \in \scrm(\mbbr_0) : \xi(x, \infty) \ge 1 \}$ for every $x > 0$. \label{nota_E_x} Then it can be checked that $\varnothing \notin {\rm cl}({\sf E}_x)$ (see the proof of Theorem~\ref{thm:ldp:rightmost:joint:regvar}) and therefore, qualifies to be a nice set. Note that $\{{\bf N}_n \in {\sf E}_1\} = \{M_n^{(1)} > \gamma_n\}$ and so, we can conclude from Theorem~\ref{thm:ldp:maxima:iid} that $r_n \prob({\bf N}_n  \in {\sf E}_x)$ converges to positive non-trivial limit as $n \to \infty$ for every $x > 0$.  But it still does not ensure that $r_n \prob({\bf N}_n \in {\sf A})$ converges to a positive limit for every nice subset ${\sf A}$ which is very much related to the second task, identification of the limit. To answer these questions, we view the sequence $(r_n \prob({\bf N}_n \in \cdot) :  n \ge 1)$ as a sequence of measures on the space $\scrm_0$ and derive its limit in the appropriate topology. We use $\mbbM_0$ topology (which will be discussed briefly in Subsection~\ref{subsec_mzero_topology}). We introduce now a few more notations. Define 
\begin{align}
\nu_\alpha(\dtv x) := \alpha q x^{-\alpha-1} \mbbo(x > 0) \dtv x + \alpha (1- q) |x|^{-\alpha -1} \mbbo(x < 0) \dtv x . \label{eq:defn:nu:alpha}
\end{align} 
and $\partial {\sf A} = {\rm cl}({\sf A}) \setminus {\rm int}({\sf A})$ denotes the boundary of the set ${\sf A}$. \label{nota_boundary_set}  \rtext{ Recall that $Z_l$ denotes the size of the $l$-th generation of the genealogical tree for every $l \ge 0$.}

\begin{thm} \label{thm:extremal:process:ldp:iid:disp}
Let the assumptions stated in Theorem~\ref{thm:ldp:maxima:iid} hold. Then there exists a non-null measure $m^*_{iid}$ on the space $\scrm(\mbbr_0)$ such that $\lim_{n \to \infty} r_n \prob({\bf N}_n \in {\sf A})   = m^*_{iid}({\sf A})$  for every ${\sf A} \in {\cal B}(\scrm_0)$ satisfying  $m^*_{iid}(\partial {\sf A}) = 0 \mbox{ and } \varnothing \notin {\rm cl}({\sf A}).$
Furthermore, 
\begin{align}
m^*_{iid}({\sf A}) = \sum_{l=0}^\infty \mu^{-l} \exptn \Big( \nu_\alpha \{ x \in \mbbr : Z_l \bdelta_x \in {\sf A}\} \Big). \label{eq:ldp:point:processes:iid:disp}
\end{align}
\end{thm}

\begin{remark}
 The first observation is that the exceptionally large value(s) form cluster as they are influenced by one large displacement. The heuristic explanation behind the limit obtained in \eqref{eq:ldp:point:processes:iid:disp} is as follows: The large displacement occurs at the $(n- l)$-th generation with probability $\mu^{-l}$ (ratio of the expected size of $(n-l)$-th generation compared to the size of $n$-th generation as each displacement has equal probability to be large) and scales up its $Z_l$ descendants. If we allow the genealogical tree ${\sf T}$ to have leaves and condition on the survival tree, then the branching random variables in $m^*_{iid}$ are replaced by their size-biased version as given in Remark~\ref{thm:ldp:pp:iid:disp:leaf}.    
\end{remark}

\begin{remark}
An important feature of the limit $m^*_{iid}$ is that the large values of extreme positions collapse in the limit. Therefore, multiple particles can stay at the same large position with positive probability. Note that this localization phenomenon is too specific to the i.i.d. displacements and changes completely even in the presence of a little bit of dependence (see Theorem~\ref{thm:main_result}). Later, we consider $(X_{i} : i \ge 1)$  a moving average process and see how this measure changes (see Theorem~\ref{thm:extreme:positions:linear:dependent:disp}). This observation motivates us to investigate further the large deviations when the displacements from the same parent are jointly regularly varying. 
\end{remark}

\subsection{Literature review} \label{subsec_review_literature}

 The investigations on the asymptotic behaviour of extreme positions were initiated in the fundamental works of Hammerseley, Kingman and Biggins. For the last two decades, the challenges, difficulties in the asymptotic study of the extreme positions have attracted many researchers. A huge literature has emerged due to its connections to the other models like first passage percolation, last passage percolation, random energy model, Gaussian free field, Gaussian multiplicative chaos.  But most of these works consider displacements with Gaussian-like tails. For a nice overview on the extreme positions for light-tailed displacements, we refer to \cite{shi:2016} and the references therein. \cite{aidekon:2013} and \cite{bramson:ding:zeitouni:2016}  are two of the fundamental works in the last decade.  The literature on the heavy-tailed displacements is very small, undeveloped compared to that and, still blooming. There are plenty of interesting and challenging open questions. Some of them are mentioned here and some of them in the Subsection~\ref{sesubsec:discussion:open:probs}. For the interest of the space, we briefly mention some of the fundamental works from the literature on extremes of BRW.

\begin{itemize}

\item 
If the displacements are assumed to have exponentially-decaying tail, the large deviations of the maximum position has recently been studied in \cite{gantert:2018} assuming the displacements to be i.i.d. The lower and the moderate deviations of maximum position has been studied in \cite{chen:he:2020}.  The large deviations for the joint behavior of the extreme positions are yet to be investigated. The large deviations for the empirical process has been studied in \cite{louidor:perkins:2015}, \cite{louidor:tsairi:2017}, \cite{chen:he:2019}, \cite{zhang:2020b} and see \cite{zhang:2020a} for the lower deviations for the level sets.

\item  In \cite{gantert:2000}, different speeds have been studied under the assumption that the displacements have semiexponential tail and are i.i.d. The fluctuations of the maximum position is recently been studied  in \cite{dyszewski:gantert:hofelsauer:2020} assuming the displacements to be i.i.d. with stretched-exponential tail.  In the same framework, the large deviations of the maximum position have been studied in \cite{dyszewski:gantert:2020}.  

\item The continuum analogue of BRW (when displacement have exponentially decaying tail) is branching Brownian motion (BBM). It is also known that the distribution of the maximum position of BBM at time $t$ satisfies famous Fisher-Kolmogorov-Petrovskii-Piskunov (FKPP) equation. The fluctuations of the extreme positions in BBM has been thoroughly studied in  \cite{aidekon:berestycki:brunet:shi:2013}, \cite{arguin:bovier:kistler:2011}, \cite{arguin:bovier:kistler:2012}, \cite{arguin:bovier:kistler:2013}. 
Large deviations of the maximum position have been studied in \cite{chauvin:rouault:1988} and the level sets have been recently obtained in \cite{aidekon:hu:shi:2017}. In \cite{derrida:shi:2016}, the large deviations for maximum position in different variants of BBM have been studied.  But the large deviations for the joint behavior of the extreme positions are yet to be studied.

\item The fluctuations of the maximum position has been studied in \cite{durrett:1979} and \cite{durrett:1983}, and weak limit of the extremal processes has been obtained in  \cite{bhattacharya:hazra:roy:2016}  assuming the displacements to be i.i.d. with regularly varying tail. \cite{bhattacharya:maulik:palmowski:roy:2016} studied the weak limit of the extremal processes for multi-type BRW where the large deviations are yet to be studied. \cite{maillard:2015} studied the tail behavior of the global maximum position when the genealogical GW tree is critical.    Large deviations for the global maximum among the positions are yet to be studied in any framework.


\end{itemize}


\section{Main results and their consequences} \label{sec:mainresults}

We start this section with a brief discussion on the $\mathbb{M}_0$ topology for the  measures on a punctured Polish space. The notations and the terminologies introduced in this discussion will help us to state  the main result rigorously. After the main result Theorem~\ref{thm:main_result}, its consequences will be discussed.  A discussion on the proof strategy and some of the open questions are given in Subsection~\ref{sesubsec:discussion:open:probs}.

\subsection{$\mathbb{M}_0$ topology for measures on Polish space} \label{subsec_mzero_topology}

Let $\mathbb{S}$ be a Polish space equipped with the metric $\rho$. The open  ball of radius $r$ centered at ${\bf x}$ will be denoted by ${\sf B}_r({\bf x}) = \{{\bf y} \in {\mathbb S} : \rho ({\bf x}, {\bf y}) < r \}$ and ${\cal B}(\mathbb{S})$ denotes the $\sigma$-algebra generated by these open balls. Fix  ${\bf s}_0 \in {\mathbb S}$ and let us consider  the subspace $\mathbb{S}_0 = \mathbb{S} \setminus \{{\bf s}_0\}$ with the $\sigma$-algebra ${\cal B}(\mathbb{S}_0)$ generated by open sets in the induced subspace topology.  ${\cal C}_+(\mathbb{S}_0)$ denotes the space of all non-negative, bounded and continuous functions on $\mathbb{S}$ which vanish in a neighbourhood ${\sf B}_r({\bf s}_0)$ for some $r > 0$. $\mathbb{M}({\mathbb S}_0)$ denotes the space of all measures $\xi$ on $\mathbb{S}_0$  such that  $\int f \dtv \xi < \infty$ for every $f \in {\cal C}_+(\mathbb{S}_0)$. A basic neighbourhood of $\xi \in \mathbb{M}(\mathbb{S}_0)$ looks like $\{\varsigma: \max_{1 \le i \le k} |\int f_i \dtv \xi - \int f_i \dtv \varsigma| < \epsilon \}$ for every finite collection $(f_i : 1 \le i \le k) \subset {\cal C}_+(\mathbb{S}_0)$. So a sub-basis can be constructed using the the sets of the form 
\alns{
\{ \xi \in \mathbb{M}(\mathbb{S}_0) : \xi(f) \in {\sf G}\}, ~~f \in {\cal C}_+ (\mathbb{S}_0) ~~ \mbox{ where } {\sf G} ~ \mbox{ is an open subset of } [0, \infty)
}
and the topology induced by this sub-basis is called $\mathbb{M}_0$ topology. \label{not_Mzero_topology} The convergence induced by this topology will be called $\mathbb{M}_0$ convergence and denoted by $\xi_n \hlconv \xi$ when $(\xi_n : n \ge 1) \in \mbbM (\mathbb{S}_0)$ converges in $\mbbM_0$ topology to the measure $\xi \in \mbbM ( \mathbb{S}_0)$. Thanks to the Portmantau theorem (Theorem~2.1 in \cite{lindskog:resnick:roy:2014}) for $\mathbb{M}_0$-convergence as it provides other alternative  characterizations of the convergence. 

In this paper, we shall be mostly concerned with $\mathbb{S} = \mbbr^d, \mbbr^\mbbn$ and $\scrm_0$ and some of the important terminologies along with fundamental facts will be recalled in the next two subsections.

\subsubsection{$\mathbb{M}_0$ convergence  in the space of point measures and the convergence determining class of functions}

Consider now $\mathbb{S} = \scrm(\mbbr_0)$ and ${\bf s}_0 = \varnothing$. It is a well-known fact that $\scrm(\mbbr_0)$ is a Polish space when equipped with the vague topology (see Proposition~3.17 in \cite{resnick:1987}). Let $\rho_v$ denote the metric induced by the vague topology and we call it {\it vague metric}. Following the previous discussion on $\mathbb{M}_0$-convergence, \label{nota_measure_on_Mzero} we can define $\mathbb{M}_0$ topology on the space $\mathbb{M}(\scrm_0)$. The functions in the convergence determining class are known in literature and  summarized in the following fact.

\begin{fact}[Lemma~A.1 and Theorem~A.2 in \cite{hult:samorodnitsky:2010}] \label{fact:HLSconv}
 Let $C_l^+(\mathbb{R}_0)$ denotes the space of Lipschitz continuous, real-valued and non-negative functions on $\mbbr_0$ with compact support. Consider two functions $g_1, g_2 \in C_l^+(\mbbr_0)$, $\epsilon_1, \epsilon_2 > 0$ and define $\hlsfun : \scrm_0 \to [0, \infty)$ as 
\aln{
\hlsfun(\nu) = \prod_{i =1}^2 \Big[ 1 - \exp \Big\{ - \Big( \int g_i \dtv \nu - \epsilon_i \Big)_+ \Big\} \Big].  \label{eq_defn_hlsfun}
} 
\begin{enumerate}

\item Let ${\tt m}_1$ and ${\tt m_2}$ be two elements in $\mbbM(\scrm_0)$. Then ${\tt m}_1 = {\tt m}_2$ if and only if $\int \hlsfun(\nu) {\tt m}_1(\dtv \nu) = \int \hlsfun(\nu) {\tt m}_2(\dtv \nu)$ for all $g_1, g_2 \in C_l^+(\mbbr_0)$ and $\epsilon_1, \epsilon_2 > 0$.

\item Let $({\tt m}_n : n \ge 1) \in \mbbM(\scrm_0)$ and ${\tt m} \in \mbbM(\scrm_0)$. Then ${\tt m}_n \hlconv  {\tt m}$ if and only if for every $g_1, g_2 \in C_l^+(\mbbr_0)$ and $\epsilon_1, \epsilon_2 > 0$, 
\alns{
\lim_{n \to \infty} \int_{\scrm_0} \hlsfun( \nu) {\tt m}_n(\dtv \nu) = \int_{\scrm_0} F_{g_1, g_2, \epsilon_1, \epsilon_2} (\nu) {\tt m}(\nu).
}  
 
\end{enumerate}
\end{fact}

\subsubsection{Regular variation on the space of all real sequences via $\mbbM_0$ topology}

For two real sequences ${\bf x} = (x_i : i \ge 1)$ and $ {\bf y} = (y_i : i \ge 1)$, define $\rho_\infty ({\bf x}, {\bf y})= \sum_{i=1}^\infty 2^{-i} (|x_i - y_i | \wedge 1)$. It is immediate that $(\mbbr^\mbbn, \rho_\infty)$ is a Polish space with $\sigma$-algebra ${\cal B}(\mbbr^\mbbn)$ generated by the open balls. Let ${\bf 0}_\infty$ denote the origin of $\mbbr^\mbbn$. We can use the framework of $\mbbM_0$ topology for the space of all measures $ \mbbM ( \mbbr^\mbbn \setminus \{{\bf 0}_\infty\}) := \mbbM(\mbbr^\mbbn_{\bf 0})$ \label{not_punctured_space_real_sequence} putting $\mathbb{S} = \mbbr^\mbbn$ and ${\bf s}_0 = {\bf 0}_\infty$. The obvious reason behind removing ${\bf 0}_\infty$ is that the regularly varying measures usually explode near the origin. We now need the concept of scaling (scalar multiplication). Define $\vartheta \cdot {\bf x} = (\vartheta x_i : i \ge 1)$ when ${\bf x} = (x_i : i \ge 1)$ for every $\vartheta > 0$. It is clear that the map $(\vartheta, {\bf x}) \mapsto \vartheta \cdot {\bf x}$ is continuous in $\vartheta$, associative and satisfies $1\cdot {\bf x} = {\bf x}$ for every ${\bf x} \in \mbbr^\mbbn$. The origin ${\bf 0}_\infty$ satisfies $\vartheta \cdot {\bf 0}_\infty = {\bf 0}_\infty$ for every $\vartheta > 0$ and  $0 < \rho_\infty({\bf x}, {\bf 0}_\infty) \le \rho_\infty({\bf 0}_\infty, \vartheta \cdot {\bf x})$ for every $\vartheta > 1$. There are many equivalent definitions for regularly varying measure (see subsection~3.2 in \cite{lindskog:resnick:roy:2014}) and  the following one among them will turn out to be very useful for us.
 
\begin{defn}[Regularly varying measures on $\mbbr^\mbbn$] \label{defn_regvar_measure}
A measure $\xi \in \mbbM (\mbbr^\mbbn_{\bf 0})$ is said to be regularly varying if there exists a measure $\varsigma \in \mbbM (\mbbr^\mbbn_{\bf 0})$ and a positive increasing and regularly varying  sequence of real numbers $(t_n : n \ge 1)$ ($\lim_{n \to \infty} t_{\lfloor n \vartheta \rfloor}/ t_n = \vartheta^\beta$ for every $\vartheta > 0$ and $\beta \ge 0$) such that $t_n \xi(n \cdot ) \hlconv \varsigma$ where $\xi(n \cdot {\sf A}) = \xi (\{n \cdot {\bf x} : {\bf x} \in {\sf A} \})$ for every ${\sf A} \in {\cal B}(\mbbr^\mbbn_{\bf 0})$.  
\end{defn}

It is known in literature (Theorem~3.1 in \cite{lindskog:resnick:roy:2014}) that the limit measure $\varsigma$ satisfies the following homogeneity property 
\aln{
\varsigma( \vartheta \cdot {\sf A}) = \vartheta^{- \alpha} \varsigma({\sf A}) \mbox{ for every } {\sf A} \in {\cal B}(\mbbr^\mbbn_{\bf 0}) \label{remark_homogeneity_limit}
}
and $ \alpha $ is called the {\it index of regular variation}. To stress the homogeneity property of the regularly varying measure $\xi$ in Definition~\ref{defn_regvar_measure}, we denote it by $\xi \in {\rm RV}_\alpha ( \mbbr^\mbbn_{\bf 0}, \varsigma) $.

\subsection{ Large deviations of the extreme positions}

%
%
%

\begin{ass}  \label{ass:joint:regvar:disp}
$\prob ({\bf X} \in \cdot) \in {\rm RV}_\alpha( \mbbr^\mbbn \setminus \{ {\bf 0}_\infty \}, \lambda_0)$ for some measure $\lambda_0 \in \mbbM(\mbbr^\mbbn_{\bf 0})$. 
\end{ass}


 Let $p_e$ \label{nota_pe} denote the probability of extinction and ${\cal S}$ denote the survival event $\bigcap_{n \ge 1}\{ Z_n \ge 1\}$. It is clear that $\prob({\cal S}) = 1 -p_e$. For any event ${\sf A}$, $\prob^*({\sf A})$ will denote the conditional probability $\prob( {\sf A}| {\cal S})$ and the corresponding conditional expectation will be denoted by $\exptn^*$. Some of the other necessary notations are as follows. 
\begin{enumerate}
\let\myenumi\theenumi
\let\mylabelenumi\labelenumi
\renewcommand{\theenumi}{N\myenumi}
\renewcommand{\labelenumi}{{\rm (\theenumi)}}

\item Let $U$ be an independent copy of $Z$. $\widetilde{U}$ denotes the random variable $U$ conditioned to stay positive i.e., $\prob(\widetilde{U} \in {\sf A}) = \prob (U \in {\sf A} | U >0)$ for every ${\sf A} \subset \mbbn$. \label{nota_u_condition_to_stay_positive}

\item For every $l \ge 1$, $(\widetilde{Z}_{l}^{(s)}: s \ge 1)$ is a collection of independent copies of the random variable $\widetilde{Z}_l$ which is the random variable $Z_l$ conditioned to stay positive. The collection $(\wt{Z}_l^{(s)} : s \ge 1, l \ge 1)$ is independent of the random variable $\wt{U}$.  \label{nota_Zl_condition_to_stay_positive}

\item Let $\pi_j : \mbbr^\mbbn \to \mbbr$ be a projection map such that  $\pi_j \Big( (x_i : i \ge 1) \Big) = x_j$ for all $j \ge 1$. \label{nota_projection_jth_coordinate}

\item We shall use $[i~:~j]$ to denote the set $\{i, i +1, \ldots, j\}$ for every pair $i < j$ of integers.   \label{nota_closed_interval_integers}

\item $|{\sf G}|$ denotes the cardinality of a set ${\sf G}$  and $\pow({\sf G})$ denotes the power set of ${\sf G}$ i.e. the collection of all subsets of ${\sf G}$ including null set $\emptyset$. \label{nota_cardinality_set}

\item Consider the measure $\lambda ({\sf A}) = \lambda_0({\sf A}) / \lambda_0({ {\cal O}})$ where ${\cal O} = \{{\bf x} \in \mbbr^\mbbn: |x_1| > 1 \}$ on $\mbbr^\mbbn_{\bf 0}$. Note that $\lambda_0({\cal  O}) < \infty$ as ${\cal O}$ is bounded away from ${\bf 0}_\infty $.    \label{defn_dashed_lamda}
\end{enumerate}

Recall (from Subsection~\ref{subsec_scope}) that the order of fluctuations of the extreme positions (at the $n$-th generation) is given by the sequence $b_n$ when Assumptions~\ref{ass_genealogy}, \ref{ass:marginal:identical:dist} and \ref{ass:joint:regvar:disp} hold. So we considered another sequence $(\gamma_n : n \ge 1)$ (see \eqref{eq:defn:gamman}) which grow faster than $b_n$. Then we introduced the  point process ${\bf N}_n $ (see \eqref{eq_scope_of_article}) which puts unit mass to each of the random variables in the collection $(\gamma_n^\inv S(\sfv) : |\sfv| = n)$. Recall also the sequence $r_n = (\mu^n \prob(|X_1| > \gamma_n) )^\inv$. It has been shown in Theorem~\ref{thm:extremal:process:ldp:iid:disp} that $r_n \prob({\bf N}_n \in {\mathscr E})$ converges to a positive finite limit when displacements are i.i.d. and ${\mathscr E} \in {\cal B}(\scrm_0)$ is chosen to be `nice'.  The following result is the generalization of the aforementioned theorem when  Assumption~\ref{ass:joint:regvar:disp} holds.


\begin{thm}[Large deviations of the extremal processes] \label{thm:main_result}
 Under the assumptions stated in \ref{ass_genealogy}, \ref{eq:assumption:branching:process} and \ref{ass:joint:regvar:disp}, we have $r_n \prob^*({\bf N}_n \in \cdot ) \hlconv m^*( \cdot )$ in the space $\mbbM(\scrm_0)$
 where for every ${\sf A} \in {\cal B}(\scrm_0)$,
\begin{align}
m^*({\sf A}) & = (1-p_e)^\inv \prob(U> 0) \sum_{l=0}^\infty \mu^{-(l+1)} \exptn \bigg[ \sum_{{\sf G} \in \pow([1~:~\widetilde{U}]) \setminus \{\emptyset\}}  \nonumber \\
& \hspace{.5cm} \lambda \Big( \mathbf{x} \in \mbbr^\mbbn : \sum_{s \in {\sf G}} \widetilde{Z}_l^{(s)} \bdelta_{x_s} \in {\sf A} \Big) \Big( \prob(Z_l >0) \Big)^{|{\sf G}|} \Big( \prob(Z_l = 0) \Big)^{\widetilde{U} - |{\sf G}|} \bigg] \label{eq:derived_lim_measure_K}.
\end{align} 
\end{thm}

We now would like to compare $m^*$ with $m^*_{iid}$ derived in Theorem~\ref{thm:extremal:process:ldp:iid:disp}. There are two new components in Theorem~\ref{thm:main_result} viz., presence of the leaves and joint regular variation the displacements. The presence of the leaves is responsible for the appearance of the conditional version of $Z_l$'s and $U$ (see Corollary~\ref{thm:extremes:jtregvar:noleaf}). The joint regular variation of the displacements have influenced the limit $m^*$ through $\lambda$ and also does not let the large values to localize (see Subsection~\ref{subsec_linearly_dep_disp}).  The limit also supports the cluster-formation phenomenon due to the random multiplicities of its atoms.  In the rest of this subsection, we shall derive consequences of Theorem~\ref{thm:main_result} when Assumption \ref{ass:joint:regvar:disp} holds.  The measure $\lambda$ appears in the limit instead of $\lambda_0$ because of the term $\prob(|X_1| > \gamma_n)$ in the scaling $r_n$ and regular variation of $\prob({\bf X} \in \cdot)$.

\begin{thm} [Large deviations of the maximum position] \label{thm:ldp:rightmost:joint:regvar}Under the assumptions stated in Theorem~\ref{thm:main_result}, we have 
\begin{align}
\lim_{n \to \infty} r_n \prob^* (M_n^{(1)} > \gamma_n x) = x^{-\alpha} {\rm C}_1 \label{eq:maxima:jtregvar}
\end{align}
for every $x > 0$ and the positive constant ${\rm C}_1$ is given by
\begin{align*}
{\rm C}_1 & = \frac{\prob(U > 0)}{1 - p_e} \sum_{l=0}^\infty \mu^{-(l+1)} \exptn \Big[ \sum_{{\sf G} \in \pow([1~:~\widetilde{U}]) \setminus \{\emptyset\}}  \big[ \lambda \big( \cup_{s \in {\sf G}} {\sf V}_s  \big) \big] \big( \prob(Z_l >0) \big)^{|{\sf G}|} \big( \prob(Z_l = 0) \big)^{\widetilde{U} - |{\sf G}|} \Big] 
\end{align*}
where  ${\sf V}_s =\big\{\pi_s^\inv (1,\infty)\big\} $ for every $s \ge 1$. 
\end{thm}

\begin{remark}[{\bf Large deviation of the minimum position and order statistics}]
Suppose that the assumptions in Theorem~\ref{thm:main_result} hold and $\wt{M}_n^{(1)} = \min_{|\sfv| = n} S(\sfv)$. Then we have $\lim_{n \to \infty} r_n \prob \big( \widetilde{M}_n^{(1)} < - \gamma_n x \big) = (1 - q) x^{- \alpha} {\rm C}_1$ where ${\rm C}_1$ is a positive constant given explicitly in Theorem~\ref{thm:ldp:rightmost:joint:regvar}.  As the large value of the extreme positions do not localize to a single value in the limit $m^*$, the large deviations of the order statistics make sense in this set-up. For simplicity, fix two positive real numbers $x_1 > x_2$. Define ${\sf E}_{x_1, x_2} = \{ \xi \in \scrm_0: \xi(x_i, \infty) \ge i \mbox{ for } i =1,2\}$. If we show that $\varnothing \notin {\rm cl}({\sf E}_{x_1, x_2})$ and $m^*(\partial {\sf E}_{x_1, x_2 }) = 0$ (see proof of Theorem~\ref{thm:extreme:positions:linear:dependent:disp}), then we can apply Theorem~\ref{thm:main_result} to obtain
\aln{
\lim_{n \to \infty} r_n \prob^* (M_n^{(1)} > \gamma_n x_1; M_n^{(2)} > \gamma_n x_2)  = \lim_{n \to \infty} r_n \prob^* \big( {\bf N}_n \in {\sf E}_{x_1, x_2} \big) = m^*({\sf E}_{x_1, x_2}). 
}
For a further simplification, we need a detailed description of the measure $\lambda$ (see Theorem~\ref{thm:extreme:positions:linear:dependent:disp}).  
\end{remark}

\begin{cor}[${\sf T}$ does not have any leaf]\label{thm:extremes:jtregvar:noleaf}
Suppose that the conditions stated in the Theorem~\ref{thm:main_result} hold. Additionally, we assume  $\prob(Z \ge 1) = 1$.  Then, we have $r_n \prob^*({\bf N}_n \in \cdot) \hlconv m^*_0(\cdot)$
 where 
\aln{
m^*_0({\sf A}) = \sum_{l =0}^\infty \mu^{- (l+1)} \exptn \Big( \lambda \big( \big\{ {\bf x} \in \mbbr^\mbbn: \sum_{s=1}^U Z_l^{(s)} \bdelta_{x_s} \in {\sf A} \big\} \big) \Big).
}
As a consequence, we have 
\aln{
\lim_{n \to \infty} r_n \prob \big( M_n^{(1)} > \gamma_n x \big) = x^{-\alpha} \sum_{l=0}^\infty \mu^{-(l+1)} \exptn \big( \lambda \big[ \bigcup_{s=1}^U {\sf V}_s \big] \big)
}
where the family of sets $({\sf V}_s : s \ge 1)$ is  introduced in Theorem~\ref{thm:ldp:rightmost:joint:regvar}.
\end{cor}


\begin{remark}[{\bf Asymptotically tail-independent displacements and ${\sf T}$ has leaves}] \label{thm:ldp:pp:iid:disp:leaf} 
We would like to discuss here how Theorem~\ref{thm:extremal:process:ldp:iid:disp} is connected to Theorem~\ref{thm:main_result}. When $X_i$'s are i.i.d. and satisfy \eqref{eq:marginal:regvar}, there exists a sequence $(t_n : n \ge 1)$ such that  $t_n \prob(n \cdot {\bf X} \in \cdot) \hlconv \lambda_{iid}(\cdot)$ in $\mbbM(\mbbr^\mbbn_{\bf 0})$ (see Subsection 4.5.1 in \cite{lindskog:resnick:roy:2014} and Subsection~3.1 in \cite{resnick:roy:2014} for a detailed discussion)  where
\aln{
 \lambda_{iid} = \sum_{l=1}^\infty \bigotimes_{j=1}^{l-1} \bdelta_0 \otimes \nu_\alpha \bigotimes_{j'= l+1}^\infty \bdelta_0. \label{eq:product:form:iid:disp}
}
 We call the displacements ${\bf X} = (X_i : i \ge 1)$ to be {\bf asymptotically tail-independent} if 
$ \prob({\bf X} \in \cdot)  \in {\rm RV}_\alpha(\mbbr^\mbbn_{\bf 0}, \lambda_{iid}).$  
Let the displacements be asymptotically tail-independent and the underlying branching process satisfy \eqref{eq:assumption:branching:process}. Then, we have $ r_n \prob^*({\bf N}_n \in \cdot ) \hlconv \wt{m}^*_{iid}(\cdot)$ in $\mbbM(\scrm_0)$ where 
\aln{
\widetilde{m}^*_{iid}(\cdot) = (1 - p_e)^\inv \sum_{l=0}^\infty \mu^{- l} \prob(Z_l > 0) \exptn \big( \nu_\alpha[ \{ x \in \mbbr_0 : \widetilde{Z}_l \bdelta_x \in \cdot \}] \big). \label{eq_ldp_pp_lim_iid_with_leaf}
}
Define $\zeta = \sum_{l =0}^\infty \mu^{- l} \prob(Z_l > 0)$. Recall that $\wt{M}_n^{(1)}$ denotes the minimum position in the $n$-th generation.  For every $x > 0$, we have 
\aln{
& \lim_{n \to \infty} r_n \prob^*(M_n^{(1)} > \gamma_n x) = \frac{\zeta q }{x^\alpha (1 - p_e)}   \mbox{ and } \lim_{n \to \infty} r_n \prob^* \big( \widetilde{M}_n^{(1)} < - \gamma_n x  \big) = \frac{\zeta (1 - q)}{ x^{ \alpha} ( 1- p_e )}. \label{eq_ldp_max_iid_with_leaf}
}
If we additionally assume $\prob(Z_1 \ge 1) = 1$, then \eqref{eq_ldp_pp_lim_iid_with_leaf} and \eqref{eq_ldp_max_iid_with_leaf} reduces to Theorem~\ref{thm:extremal:process:ldp:iid:disp} and \ref{thm:ldp:maxima:iid}. For the detailed derivation of \eqref{eq_ldp_pp_lim_iid_with_leaf} and \eqref{eq_ldp_max_iid_with_leaf}, see proofs of Theorem~\ref{thm:extremal:process:ldp:iid:disp} and \ref{thm:ldp:maxima:iid}. 
\end{remark}

\begin{cor}[ ${\rm b}$-ary genealogical tree] Suppose that $\prob(Z = {\rm b} ) =1$ in definition of $\scrl$ (see \eqref{defn:progeny:pp:BRW}) and the displacements are asymptotically tail-independent. Then we can modify the definition of $r_n$ and $\gamma_n$ with $\mu = {\rm b}$ to see that 
\aln{
\lim_{n \to \infty} r_n \prob( {\bf N}_n \in {\sf B}) = \sum_{l=0}^\infty {\rm b}^{- l} \nu_\alpha (\{x \in \mbbr: {\rm b}^l \bdelta_x \in {\sf B}\})
} 
for every nice subset ${\sf B} \in {\cal B}(\scrm_0)$. As a consequence, we can derive 
\aln{
& \lim_{n \to \infty} r_n \prob(M_n^{(1)} > \gamma_n x) = \frac{q  {\rm b}}{x^\alpha ( {\rm b} -1)} \mbox{ and } \lim_{n \to \infty} r_n \prob( \widetilde{M}_n^{(1)} < -  \gamma_n x) = \frac{(1 - q) {\rm b}}{x^{ \alpha} ({\rm b}-1)}.
} 
\end{cor}

\subsection{Dependent displacements} \label{subsec_linearly_dep_disp}

Here we consider a simple case where the displacements from the same parent are assumed to be linearly dependent. Consider a positive random variable $Y$ such that 
$\prob(Y > y) = y^{- \alpha} {\rm L}(y) \mbox{ for all } y > 0$
where ${\rm L}$ is a slowly varying function. It is clear that  $\prob(Y \in \cdot) \in {\rm RV}_{- \alpha}((0, \infty), \nu_\alpha^+)$  such that $\nu_\alpha^+(x, \infty) = x^{- \alpha}$ for every $x > 0$.  Consider the collection ${\bf Y} = (Y_0, Y_1, Y_2)$ of independent copies of the random variable $Y$ and define 
\alns{
X_i = \phi Y_{i -1} + Y_i \mbox{ for } i =1,2  
}
for some deterministic constant $\phi > 1$.  We now consider the BRW constructed based on $ \scrl_\phi = \sum_{i =1}^2 \bdelta_{X_i}$. It is clear that each particle produces exactly two children and the displacements (from the same parent) form ${\rm MA}(2)$ process with deterministic coefficient. It can be shown that ${\bf X} = (X_1, X_2)$ is jointly regularly varying on the space $(0, \infty)^2$.  We choose a sequence $(\gamma_n^{(\phi)}: n \ge 1)$ such that 
\aln{
\lim_{n \to \infty} 2^n \prob(Y > \gamma_n^{(\phi)}) = 0. \label{eq:defn:gamma_n:phi}
}
Define
$r_n^{(\phi)} = (2^n \prob( Y > \gamma_n^{(\phi)}))^\inv$ and  ${\bf N}_n^{(\phi)}: = \sum_{|\sfv| = n} \bdelta_{(\gamma_n^{(\phi)})^\inv S(\sfv)}. 
$
 We derive the large deviations for the extreme positions in the following theorem.

\begin{thm} \label{thm:extreme:positions:linear:dependent:disp}
Consider a BRW with the point process $\scrl \eqd \scrl_\phi$. There exists a measure $m^*_\phi \in \mbbM(\scrm_0)$ such that $r_n^{(\phi)} \prob \big( {\bf N}_n^{(\phi)} \in \cdot \big) \hlconv m^*_\phi(\cdot)$ where  for every ${\sf A} \in {\cal B}(\scrm_0)$,
\aln{
m^*_\phi({\sf A}) & =  \sum_{l=0}^\infty 2^{- l } \Big[(1 + \phi^\alpha) \nu_\alpha^+ \big( \{ z >0 : 2^l \bdelta_z \in {\sf A} \} \big)  +  \nu_\alpha^+ \big( \{ z > 0 : 2^l (\bdelta_z + \bdelta_{\phi z}) \in {\sf A} \}  \big) \Big]. \label{eq:expression:limit:measure:linearly:dependent:disp}
}
Furthermore, for $0 < x_2 < x_1$,  
\aln{
& \lim_{n \to \infty} r_n^{(\phi)} \prob \big( \{ M_n^{(1)} >  \gamma_n^{(\phi)} x_1 \} \cap \{ M_n^{(2)} >  \gamma_n^{(\phi)} x_2 \} \big) =  \frac{ (1 + 2\phi^\alpha)} {x_1^{ \alpha}} + (\max(x_2,  x_1/\phi))^{- \alpha}  \label{eq_two_largest_order_stat_lin_dep} \\
& \mbox{and } ~\lim_{n \to \infty} r_n^{(\phi)} \prob( M_n^{(1)} > \gamma_n^{(\phi)} x_1) = (1 + 2 \phi^\alpha) x_1^{- \alpha}.  \label{eq_max_lin_dep}
}
\end{thm}

\begin{cor}[{\bf Asymptotic full dependence}] Suppose that $X_i = X_1$ for every $i \ge 1$. Then we have $\scrl \eqd Z \bdelta_{X_1}$ (see page~195 in \cite{resnick:2007}).  This means that the displacements from the same parent are the same.  Then it is clear that $\prob({\bf X} \in \cdot) \in {\rm RV}_\alpha(\mbbr^\mbbn_{\bf 0}, \lambda_f)$ where $\lambda_f (\cdot) = \nu_\alpha(x: ~x. {\bf 1}_\infty \in \cdot )$ and ${\bf 1}_\infty$ is  sequence with all elements equal to $1$. If we use this fact in Theorem~\ref{thm:main_result}, then we have 
 \aln{
  \lim_{n \to \infty } r_n \prob^* \big( {\bf N}_n \in {\sf A} \big) & =  \frac{\prob(U > 0)}{(1 - p_e)} \sum_{l = 0}^\infty \mu^{-(l + 1)} \exptn \Big[ \sum_{k =1}^{\wt{U}} \binom{\wt{U}}{k} \big( \prob[Z_l > 0] \big)^k \big( \prob(Z_l = 0) \big)^{\wt{U} - k}  \nonumber \\
& \hspace{2cm} \nu_\alpha \big( \{x \in \mbbr: \bdelta_x \sum_{i= 1}^k  \wt{Z}_l^{(i)} \in {\sf A}\} \big) \Big]  \label{eq_ldp_lim_measure_full_dep}
}
Recall that ${\sf E}_x = \{\xi \in \scrm : \xi(x, \infty) \ge 1\}$. It will be shown in the proof of Theorem~\ref{thm:ldp:rightmost:joint:regvar} that $\varnothing \notin {\rm cl}({\sf E}_x)$ and the limit measure puts zero mass on the boundary $\partial {\sf E}_x$. Therefore, after some algebra, we have
\alns{
& \lim_{n \to \infty} r_n \prob^*(M_n^{(1)} \ge \gamma_n x) = \lim_{n \to \infty} r_n \prob^* \big( {\bf N}_n \in {\sf E}_x \big) =\frac{q\prob(U > 0)}{(1 - p_e) x^{ \alpha}} \sum_{l =0}^\infty \mu^{- (l + 1)} \exptn \Big[ 1 - \big( \prob[Z_l = 0] \big)^{\wt{U}} \Big]
}
as $\wt{Z}_l^{(i)} \ge 1$ almost surely for every $l \ge 0$ and $i \ge 1$.  
\end{cor}

\subsection{Discussion and open questions} \label{sesubsec:discussion:open:probs}

The proof of Theorem~\ref{thm:main_result} heavily uses the modern tools developed to study the large deviations of the point processes. The proof has been divided into five main steps which are (i)~at the most one large displacement on a geodesic path, (ii)~cutting the tree, (iii)~one subree in the forest suffices, (iv)~pruning the subtree and (v)~regularization.  We did not change the names of the four steps from \cite{bhattacharya:hazra:roy:2017} only because of the similarity in the modification of the genealogical structure at each of these step.  The step `one subtree in the forest suffices' is new where we show that the probability of the large displacements in the forest is proportional to the that of a subtree (see Propositions~\ref{lemma:one:tree:large:disp}  and \ref{lemma:one:subtree}). It does not need a special mention that the principle of `a bunch of large displacements' turns out to be the thread which binds these five steps together. We now mention a few open questions which are not answered in this study.

Our study depends heavily on the Kesten-Stigum condition. It would be natural to see how this picture changes when Kesten-Stigum condition is violated.  For example, let us consider the case where the progeny mean may be finite but, Kesten-Stigum condition may not hold. In that case, one has to use the Seneta-Heyde scaling sequence to normalize $(Z_n : n \ge 1)$ (see Theorem~3 at page~30 in \cite{athreya:ney:1972}). In this case, the limit random variable $W$ does not have a finite expectation which has been used in the proof of Proposition~\ref{lemma:one:subtree}.  We have a strong belief that the whole picture might be different. Another important case would be where the progeny mean is infinite. Recently, in \cite{ray:hazra:roy:soulier:2020}, the weak limit of the extremal processes has been obtained and it is proved that the limit does not have any cluster/decoration (counter-intuitive to the predictions in \cite{brunet:derrida:2011}). We believe that the study of the large deviations of extreme positions in this framework will be interesting and exciting as may lead to a different limit measure.

\section{Proof of Theorem~\ref{thm:main_result} }\label{sec:proofs}

\subsection{Proof of Theorem~\ref{thm:main_result} based on auxiliary results}

 In view of Fact~\ref{fact:HLSconv}, it is enough to establish
\aln{
 &  \lim_{n \to \infty} r_n \exptn^* \Big(F_{g_1, g_2, \epsilon_1, \epsilon_2} ({\bf N}_n) \Big)  = \int_{\scrm_0} F_{g_1, g_2, \epsilon_1, \epsilon_2}(\nu) m^*(\dtv \nu) \label{eq:aim:theorem:main:result}
}
to prove Theorem~\ref{thm:main_result}. This will be achieved through the following approximation steps. Recall that $X(\sfv)$ and $S(\sfv)$ denote the displacement and the position respectively of the vertex $\sfv$. ${\sf I}(\sfv)$ \label{nota_Iv} is used to denote the geodesic path ${\varpi} \mapsto \sfv $ where ${\varpi}$ denotes the root of the genealogical tree ${\sf T}$.  The first proposition connects local extreme positions to the global extreme displacements.

\begin{propn}[At the most one large displacement can occur on a typical geodesic path]\label{lemma:one_large_jump}
Define 
\aln{
\widetilde{\bf N}_n = \sum_{|\sfv| = n} \sum_{\sfu \in {\sf I}(\sfv)} \bdelta_{\gamma_n^\inv X(\sfu)}. \label{eq_defn_wtNn}
}
Under the assumptions stated in Theorem~\ref{thm:main_result}, for every $g_1, g_2 \in \ccr$ and $\epsilon_1, \epsilon_2 >0$, we have
\aln{
\lim_{n \to \infty} r_n \big| \exptn^* \Big(  F_{g_1,g_2, \epsilon_1, \epsilon_2}({\bf N}_n)  \Big)- \exptn^* \Big(  F_{g_1, g_2, \epsilon_1, \epsilon_2}(\wt{\bf N}_n) \Big) \big| = 0.
}
\end{propn} 

We then cut the tree at the $(n - \kappa)$-th generation and create a forest. The height of the each of the subtrees in the forest is at the most $\kappa$. We use ${\sf u} \mapsto {\sf v}$ to denote the geodesic path from the vertex $\sfu$ to its descendant $\sfv$ including $\sfu$. \label{nota_geodesic_path}   Define 
\aln{
\widetilde{\bf N}_{n, \kappa} = \sum_{|\sfv| = n} ~~\sum_{\sfu \in {\sf I}({\sfv}) : |\sfu \mapsto \sfv| \le \kappa - 1} \bdelta_{\gamma_n^\inv X(\sfu)}. 
}
Note that $\wtbfN_{n, \kappa}$ does not contain any information about the first $(n - \kappa)$ generations. In the next proposition, we show that the last few generations can only contain the large displacements with high probability.

\begin{propn}[The forest is most likely to contain the large displacements] \label{lemma:cutting:tree}

Under the assumptions stated in Theorem~\ref{thm:main_result}, for every $g_1, g_2 \in \ccr$ and $\epsilon_1, \epsilon_2 > 0$, we have
\aln{
\lim_{\kappa \to \infty}\limsup_{n \to \infty} r_n \big| \exptn^* \big[F_{g_1,g_2,\epsilon_1, \epsilon_2} (\wtbfN_n) - F_{g_1, g_2, \epsilon_1, \epsilon_2}(\wt{\bf N}_{n , \kappa}) \big] \big| = 0. \label{eq:cutting:tree:aim}
}
\end{propn}

Recall that ${\sf D}_i$ denotes the collection of vertices in the $i$-th generation and $Z_i$ denotes the cardinality of ${\sf D}_i$ for every $i \ge 0$. The forest will be denoted by $(\sfT^{(i)}_\kappa : 1 \le i \le Z_{n- \kappa})$ using the lexicographic order of the Ulam-Harris labels of the vertices in ${\sf D}_{n - \kappa}$.   Note that the displacement of the root of each of the subtrees in the forest is zero according to our convention. Let ${\varpi }_i$ denote the root of the subtree ${\sf T}_\kappa^{(i)}$ for $1 \le i \le Z_{n - \kappa}$. These notations lead to 
\aln{
\widetilde{\bf N}_{n, \kappa} =\sum_{i =1}^{Z_{n - \kappa}} \sum_{\sfu \in \sfT_\kappa^{(i)} \setminus \{{\varpi}_i\}} A(\sfu) \bdelta_{\gamma_n^\inv X(\sfu)} \label{eq:alt:expression:N:nkappa}
}
where $A {(\sfu)}$ \label{nota_Au} denotes the number of descendants of the vertex $\sfu$ in the $\kappa$-th generation of the tree $\sfT^{(i)}_\kappa$. If the $\kappa$-th generation of the subtree $\sfT^{(i)}_\kappa$ is empty (due to leaves), then we define $A(\sfu) = 0$ for all $\sfu \in \sfT^{(i)}_\kappa$. 


\noindent{\bf One subtree with large displacement is responsible for the large deviations.} We shall now show that at the most one of the subtrees contain the large displacement with high probability. So, it is enough to study large displacements in this tree and \rtext{count} their descendants with special care.  This is the most important observation and very specific for the large deviations. The proof heavily relies on the fact that the forest contains i.i.d. subtrees.  


 Define $\varphi = \min (\varphi_1, \varphi_2)$ where $\varphi_j$'s are chosen in such a fashion that ${\rm support}(g_j) = \{|x| : g_j(x) > 0\} \subset (\varphi_j, \infty)$ for $j =1,2$.  Define the event 
\aln{
{\sf H}_n^{(i)} := \Big\{ \sum_{\sfu \in {\sf T}_\kappa^{(i)} \setminus \{{\varpi }_i\}} \bdelta_{ \gamma_n^\inv |X(\sfu)| } (\varphi/2, \infty) \ge 1 \Big\} \mbox{ for every  } i \in [1: Z_{n - \kappa}]. 
}
Define another event
\aln{
 {\sf Q}_n^{(1)} := \bigcup_{i = 1}^{Z_{n - \kappa}} \Big\{ {\sf H}_n^{(i)} \bigcap \bigcap_{i' \in [1:~Z_{n - \kappa}] \setminus \{i \}} \big( {\sf H}_n^{(i')} \big)^c   \Big\} \label{eq_defn_qn_one}
}
which denotes that exactly one subtree can contain the large displacements. Finally, we define the events 
\aln{
{\sf Q}_n^{(0)} = \bigcap_{i  \in [1: ~{Z}_{n - \kappa}]} ({\sf H}_n^{(i)})^c \mbox{ and } {\sf Q}_n^{(\ge 2)} = ({\sf Q}_n^{(0)} \cup {\sf Q}_n^{(1)})^c. \label{eq_defn_qn_two}
}  

\begin{lemma}[Large displacements are most likely to occur in at most one of the subtrees] \label{lemma:one:tree:large:disp}
Under the assumptions stated in Theorem~\ref{thm:main_result}, for every $\kappa \ge 1$, we have $\lim_{n \to \infty} r_n \prob^* ({\sf Q}_n^{(\ge 2)}) = 0$. 
\end{lemma}

For every $i  \in[1~:~ Z_{n - \kappa}]$, define 
\aln{
\wtbfN_{n, \kappa}^{(i)} = \sum_{\sfu \in \sfT_\kappa^{(i)}} A(\sfu) \bdelta_{\gamma_n^\inv X(\sfu)}.
}
It is then clear that the point process $\wtbfN_{n, \kappa}$ is superposition of the i.i.d. point processes attached to the forest $(\sfT_\kappa^{(i)}:  i \in [1~:~ Z_{n - \kappa}])$. Additionally, it can be shown that the point process associated to each of the subtrees have regularly varying tail using $\mbbM_0$ convergence on the space $\mbbM(\scrm_0)$.  Therefore, it is natural to guess that \rtext{at most} one of them contain a bunch of the large displacements with high probability. We can call this phenomenon as ``principle of a point process with large atoms'' and implies that it is sufficient to study the large deviations of the point process associated with a typical  subtree  and simplifies our analysis.  

\begin{propn}[One subtree with large displacements suffices] \label{lemma:one:subtree}
  Under the assumptions stated in Theorem~\ref{thm:main_result}, for every $\kappa \ge 1$ and $\varrho \ge 1$, we have
\alns{
\lim_{n \to \infty} \frac{ r_n \exptn^* \Big( F_{g_1,g_2, \epsilon_1,\epsilon_2} \big[\wtbfN_{n, \kappa}\big] \Big) }{  \exptn \Big(  F_{g_1,g_2, \epsilon_1, \epsilon_2} \big[ \wtbfN_{n, \kappa}^{(1)} \big]  \Big) / \prob(|X_1| > \gamma_n) } = \mu^{- \kappa} (1 - p_e)^\inv.
}
\end{propn}

In the next step, we shall prune the subtree $\sfT_\kappa^{(1)}$. Let $\varrho$ be a large integer satisfying $\exptn(Z_1 \wedge \varrho) > 1$. Then we start form the root and modify the descendants of each vertex $\sfv \in {\sf T}_\kappa^{(1)}$, by keeping the first $\varrho$ children (honoring the lexicographic order of their Ulam-Harris labels) and deleting the    others with their line of descendants. The pruned version of the subtree will be denoted by $\mathbb{T}_{\kappa}^{(1)}$\label{not_pruned_subtree}. Define 
\aln{
\wtbfN_{n, \kappa, \varrho}^{(1)} = \sum_{\sfu \in \mbbt_\kappa^{(1)} \setminus \{{\varpi}_1\}} A^{(\varrho)}(\sfu) \bdelta_{\gamma_n^\inv X(\sfu)}.
}
where $A^{(\varrho)}(\sfu )$ denotes the number of displacements of $\sfu$ in the $\kappa$-th generation of $\mathbb{T}_\kappa^{(1)}$. 


\begin{propn}[The pruned subtree contains the large displacement with high probability] \label{propn_pruned_subtree}
If the assumptions in Theorem~\ref{thm:main_result} holds, then for every $\kappa \ge 1$, we have
\aln{
\lim_{\varrho \to \infty} \limsup_{n \to \infty} \big[ \prob(|X_1| > \gamma_n) \big]^\inv \exptn \Big| \hlsfun(\wtbfN_{n, \kappa}^{(1)}) - \hlsfun(\wtbfN_{n, \kappa, \varrho}^{(1)}) \Big| = 0.
}
\end{propn}


In the step `regularization', we make the subtree $\mbbt_\kappa^{(1)}$ regular by adding necessary child/children to a vertex if it has less than $\varrho$ children. The regularized version of $\mbbt_\kappa^{(1)}$ will be denoted by $\scrt_\kappa^{(1)}$.\label{not_regular_tree}  If $\sfu \in \scrt_\kappa^{(1)} \setminus \mbbt_\kappa^{(1)}$, then we define $A^{(\varrho)}(\sfu) = 0$. We also attach new displacement $X'(\sfu)$ to each vertex $\sfu \in \scrt_\kappa^{(1)}$ such that the displacements of $(X'({(\sf u,i)}):~ i \in [1~:~\varrho]) $ is an independent copy of the random vector ${\bf X}_{|\varrho} = (X_i : i \in [1:\varrho])$.  It then follows that 
\aln{
\wtbfN_{n, \kappa, \varrho}^{(1)} \eqd \sum_{\sfu \in {\scrt}_\kappa^{(1)} \setminus \{{\varpi}_1\}} A^{(\varrho)}(\sfu) \bdelta_{\gamma_n^\inv X'(\sfu)}.
}
We denote the point processes in the right hand side by $\wtbfN_{n, \kappa, \varrho}^{(1)}$ with a slight abuse of notation. In the following proposition, we derive and identify explicitly the limit of $\exptn(\hlsfun(\wtbfN_{n, \kappa, \varrho}^{(1)}))/ \prob(|X_1| > \gamma_n)$ which completes the proof of Theorem~\ref{thm:main_result}.

\begin{propn}[Emergence of the limit measure $m^*$] \label{lemma:computing:limit}
If the assumptions in Theorem~\ref{thm:main_result} hold, then we have
\aln{
& \lim_{\kappa \to \infty} \lim_{\varrho \to \infty} \lim_{n \to \infty}  \frac{\exptn \big( F_{g_1, g_2,\epsilon_1, \epsilon_2}(\wt{\bf N}^{(1)}_{n, \kappa, \varrho}) \big)}{ \mu^{\kappa} \prob (|X_1| > \gamma_n) }  = (1 - p_e) \int_{\scrm_0} F_{g_1,g_2, \epsilon_1, \epsilon_2}( \nu) m^*(\dtv \nu).  \label{eq:propn:comp:lim}
} 
\end{propn}



\section{Proofs of the auxiliary results}

\subsection{Proof of Proposition~\ref{lemma:one_large_jump}} \label{subsec:one_large_jump}

We shall first show that there could be at the most one large displacement on a typical geodesic path by establishing 
\aln{
\lim_{n \to \infty} r_n \prob^*({\sf G}_n^c) = 0 \mbox{ where } {\sf G}_n = \Big( \bigcup_{|\sfv | = n}  \Big\{ \sum_{\sfu \in {\sf I}(\sfv) } \bdelta_{|X(\sfu)|} (\gamma_n n^{-3}, \infty) \ge 2 \Big\}\Big)^c.  \label{eq:lemma_olj_amo_1}
}
Then consider two measures $\xi_1$ and $\xi_2$ living on the space $\mbbr_0$. It follows from basic algebra and the inequality $|e^{- x} - e^{- y}| \le |x - y|$ for all  $x, y > 0$ that 
\aln{
\big| \hlsfun(\xi_1) - \hlsfun(\xi_2) \big| \le 2 \wedge \Big[ 2 \sum_{i =1}^2 |\xi_1(g_i) - \xi_2(g_i)|  \Big]. \label{eq_comp_ineq_hlsfun_two_measures}
}
This inequality combined with the claim in \eqref{eq:lemma_olj_amo_1} implies that 
\alns{
r_n \exptn^* \Big[ \Big| \hlsfun(\wtbfN_n) - \hlsfun({\bf N}_n)\Big| \mbbo({\sf G}_n^c) \Big]  = o(1). 
}
If we use the inequality in \eqref{eq_comp_ineq_hlsfun_two_measures} for the point processes $\wtbfN_n$ and ${\bf N}_n$ on the event ${\sf G}_n$, then the claim of the proposition reduces to 
\aln{
\limsup_{n \to \infty} r_n \exptn^* \Big[ |\wtbfN_n(g) - {\bf N}_n(g) | \mbbo({\sf G}_n) \Big] = 0 \label{eq_olj_main_aim_comp}
}
for every $g \in \ccr$. 

\noindent{\bf Proof of \eqref{eq_olj_main_aim_comp}.} As the probability of survival is positive, replacing $\exptn$ by $\exptn^*$ will not affect the asymptotics of the conditional expectation in \eqref{eq_olj_main_aim_comp}.  Let $\wt{\varphi} > 0$  be chosen in such \rtext{a way} that $\inf\{ |x| : g(x) > 0 \} \subset (\wt{\varphi}, \infty)$ and $T(\sfv)$ be the maximum of the absolute values of the displacements on the path ${\sf I}(\sfv)$ when $\sfv \in {\sf D}_n$. On the event $\{ \max_{|\sfv| = n} T(\sfv) < \gamma_n \wt{\varphi} /2 \} \cap {\sf G}_n$, it follows that $\wtbfN_n(g) = 0$ almost surely for large enough $n$. On the same event, it also follows that $|S(\sfv) - T(\sfv)| \le  \gamma_n n^{-2}$ for every $n \ge 1$ which implies that ${\bf N}_n(g) = \sum_{|\sfv| = n} g(\gamma_n^\inv S(\sfv)) = 0$ almost surely for large enough $n$. For large enough $n$, we then have \rtext{ the } following almost sure upper bound for the random variable inside expectation in \eqref{eq_olj_main_aim_comp}
\aln{
\sum_{|\sfv| = n} \Big| g \big[ \gamma_n^\inv S(\sfv) \big] - g \big[ \gamma_n^\inv T(\sfv) \big] \Big| \mbbo\Big( \{\max_{|\sfv| = n}T(\sfv) > \gamma_n \wt{\varphi}/2 \} \cap {\sf G}_n \Big). \label{eq_olj_Sv_Tv_upper_bound}
}
As $g$ is bounded and continuous, $g$ is uniformly continuous on the compact set $\{x \in \mbbr : |x| \ge  \wt{\varphi} \}$ and let $\norm{g}$ denote the modulus of uniform continuity of the function $g$.  This fact yields the following almost sure upper bound for \eqref{eq_olj_Sv_Tv_upper_bound}
\aln{
& \norm{g} \gamma_n^\inv \sum_{|\sfv| = n }  \big| S(\sfv) - T(\sfv) \big| \mbbo \big( \{ \max_{|\sfv| = n} T(\sfv) > \gamma_n \wt{\varphi}/2 \} \cap {\sf G}_n \big) \nonumber \\
&=  \norm{g} n^{-2}  Z_n \mbbo \Big( \big\{ \max_{|\sfv| = n} T(\sfv) > \gamma_n \wt{\varphi}/2 \big\} \cap {\sf G}_n  \Big)  \le \norm{g}  n^{-2} Z_n \mbbo \big( \{ \max_{|\sfv| = n} T(\sfv) > \gamma_n \wt{\varphi}/2 \} \big). \label{eq_olj_Tv_upper_bound}
}
Observe now that $T(\sfv)  \eqd \max_{1 \le i \le n} |X_1^{(i)}|$ for every $\sfv \in {\sf D}_n$ where $(X_1^{(i)} : i \ge 1)$ are independent copies of the random variable $X_1$ using the fact that the displacements on a path are i.i.d. Then the union bound for $T(\sfv)$ in \eqref{eq_olj_Tv_upper_bound} combined with the independence of branching mechanism and displacements, yields
\aln{
r_n \exptn \big[ \big| \wt{\bf N}_n(g) - {\bf N}_n(g) \big| \big] \le \norm{g} n^\inv  \frac{\prob(|X_1| > \gamma_n \wt{\varphi}/2)}{\prob(|X_1| > \gamma_n)} = o(1).
} 
So the proof is complete. 


\noindent{\bf Proof of \eqref{eq:lemma_olj_amo_1}.} As the probability of survival is positive, we can replace $\prob^*$ by $\prob$ in \eqref{eq:lemma_olj_amo_1} without affecting its asymptotic behavior. Let $(X_1^{(i)} : i \ge 1)$ be a collection of i.i.d. copies of the random variable $X_1$. Using union bound for the event ${\sf G}_n^c$ combined with the independence between branching mechanism and displacements, and $\sum_{\sfu \in {\sf I}(\sfv)} \bdelta_{\gamma_n^\inv |X({\sfu})|} \eqd \sum_{|\sfv| = n} \bdelta_{\gamma_n^\inv |X_1^{(i)}|}$ for every $\sfv \in {\sf D}_n$; we get following upper bound for $\prob({\sf G}_n^c)$
\aln{
\mu^n \prob \big( \sum_{i =1}^n \bdelta_{\gamma_n^\inv |X_1^{(i)}|} (n^{-3}, \infty) \ge 2 \big) \sim {\rm C}_1 \mu^n n^2 \big[ \prob(|X_1| > \gamma_n n^{- 3}) \big]^2 \label{eq_asymptotic_order_amo_event}
}
where ${\rm C}_1$is some positive constant. For the derivation of the \rtext{asymptotic} order of the probability in the left hand side of \eqref{eq_asymptotic_order_amo_event}, we have used the fact that $\prob({\mathscr B}_n \ge 2) = O(n^2 p_n^2)$ where $\mathscr{B}_n $ is ${\rm Binomial}(n,p_n)$  with $np_n = O(1)$. The upper bound obtained in \eqref{eq_asymptotic_order_amo_event} leads to the following upper bound for $r_n \prob( {\sf G}_n^c)$ 
\aln{
{\rm C}_1 ~\big[ n^2 \prob(|X_1| > \gamma_n) \big] \Big[ \frac{\prob(|X_1| > \gamma_n n^{-3} \wt{\varphi}/2)}{\prob(|X_1| > \gamma_n)} \Big]^2 \label{eq_upper_bound_amo}
}
The ratio can be dealt with the help of Potter's bound  (see Proposition~0.8(ii) in \cite{resnick:1987}) as $n^{-3} \gamma_n \to \infty$. Fix an $\eta \in (0,1)$. Then Potter's bound implies that the ratio is of order $ O(n^{3( \alpha + \eta) })$. As the probability in the first term of \eqref{eq_upper_bound_amo} is decaying exponentially first, the expression converges to zero as $n \to \infty$ and the proof is complete.

\subsection{Proof of Proposition~\ref{lemma:cutting:tree}}

Fix $\theta > 0$. Define 
\aln{
{\sf R}_{n, \kappa}(\theta) = \{ \sum_{|\sfu| \le n - \kappa} \bdelta_{\gamma_n^\inv |X(\sfu)|} (\theta, \infty) \ge 1\}. \label{eq_defn_cutting_event}
}
Our first claim is that for every $\theta > 0$,
\aln{
\lim_{\kappa \to \infty} \limsup_{n \to \infty} r_n \prob^* \big[{\sf R}_{n, \kappa}(\theta) \big] = 0. \label{eq:reduced:aim:first:cutting}
}
Recall that $\varphi_i = \inf \{ |x| : g_i(x) > 0 \}$ for $i =1,2$ and $\varphi = \delta_1 \wedge \delta_2$. Then putting $\theta = \varphi/2$, we get $\exptn^* \big[ F_{g_1, g_2, \epsilon_1, \epsilon_2}(\wt{\bf N}_n)$ $\mbbo \big(\{{\sf R}_{n , \kappa}(\varphi/2)\}^c \big) \big]  = \exptn^* \big[ F_{g_1, g_2, \epsilon_1, \epsilon_2}\big(\wt{\bf N}_{n, \kappa} \big)  \big]$  Hence, the proof of the proposition is completed assuming \eqref{eq:reduced:aim:first:cutting}.


\noindent{\bf Proof of \eqref{eq:reduced:aim:first:cutting}.} As the probability of survival is positive, we can replace $\prob^*$ by $\prob$ without harming its rate of decay.
Note that $\sum_{|\sfu| \le n- \kappa} \bdelta_{X(\sfu)} =  \sum_{|\sfu| \le n- \kappa -1} \scrl_{\sfu}$ where ${\scrl}_{\sfu} = \sum_{i =1}^{Z(\sfu)} \bdelta_{X_i(\sfu)} $ is the independent copy of the point process $\scrl$ produced by $\sfu$. \label{nota_Zu} It is clear that 
\alns{
{\sf R}_{n, \kappa}(\theta) = \Big\{ \sum_{|\sfu |\le n-\kappa - 1}  \scrl_\sfu  \big(\{ \gamma_n x : |x | > \theta \} \big) \ge 1 \Big\} = \bigcup_{|\sfu| \le n- \kappa - 1} \bigcup_{i =1}^{Z(\sfu)} \{ |X_i(\sfu)| > \gamma_n \theta\}.
}
Recall that the displacements are identically distributed and independent of the branching mechanism. So union bound yields the following upper bound
\alns{
r_n \exptn(Z_1 + Z_2 + \ldots + Z_{n - \kappa}) \prob(|X_1| > \gamma_n \theta) \le \frac{\mu}{\mu - 1} \mu^{- \kappa} \frac{\prob(|X_1| > \gamma_n \theta)}{\prob(|X_1| > \gamma_n )}.
} 
 The ratio of the probabilities converges to $\theta^{- \alpha}$ because of \eqref{ass:marginal:identical:dist} when $n \to \infty$ and \eqref{eq:reduced:aim:first:cutting} follows immediately if we let $\kappa \to \infty$. This completes the proof.

\subsection{Proof of Lemma~\ref{lemma:one:tree:large:disp}}

Recall the definitions of ${\sf Q}_n^{(0)}$, ${\sf Q}_n^{(1)}$ and ${\sf Q}_n^{(\ge 2)}$ from \eqref{eq_defn_qn_one} and \eqref{eq_defn_qn_two}. It follows that $\prob ( {\sf Q}_n^{\ge 2}) = 1 - \prob({\sf Q}_n^{(0)}) - \prob({\sf Q}_n^{(1)}) $.  Define 
\alns{
\ol{\bf N}_{n, \kappa}^{(j)} = \sum_{\sfu \in {\sf T}_{\kappa}^{(j)} \setminus \{{\varpi}_j\}} \bdelta_{\gamma_n^\inv |X(\sfu)|} 
}
where ${\varpi}_j$ is the root of the subtree ${\sf T}_\kappa^{(j)}$  for every $j =1,2, \ldots, Z_{n - \kappa}$.  It follows from the assumptions that  $\wt{\bf N}_{n, \kappa}$ is the superposition of i.i.d. point processes $(\ol{\bf N}_{n, \kappa}^{(j)} : j \in [1: Z_{n - \kappa}])$ conditioned on $Z_{n - \kappa}$. We use these observations together combined with the inequality $1 - x^n \le n( 1- x)$ for all $x \in (0,1)$ and $n \ge 1$ to obtain the following upper bound for $\prob({\sf Q}_n^{\ge 2})$ 
\aln{
& \exptn \Big[ 1 - \Big( \prob \big( \ol{\bf N}_{n, \kappa}^{(1)}(\varphi/2, \infty) = 0 \big) \Big)^{Z_{n - \kappa}} - Z_{n - \kappa} \prob(\ol{\bf N}_{n, \kappa}^{(1)} (\varphi/2, \infty) \ge 1)  \Big( \prob \big( \ol{\bf N}_{n, \kappa}^{(1)}(\varphi/2, \infty) = 0 \big) \Big)^{Z_{n - \kappa} - 1}  \Big] \nonumber \\
& \le \exptn \Big( Z_{n - \kappa} \prob \big( \ol{\bf N}_{n, \kappa}^{(1)} (\varphi/2, \infty) \ge 1 \big) \Big[ 1- \Big( \prob \big( \ol{\bf N}_{n, \kappa}^{(1)}(\varphi/2, \infty)  = 0 \big) \Big)^{Z_{n - \kappa} - 1}  \Big] \Big). \label{eq_one_subtree_binomial_upper_bound}
}
To study the asymptotic order of the expectation, we shall truncate $Z_{n - \kappa}/ \mu^{n - \kappa}$. Fix a large positive integer ${\rm K}$. If we use the inequality $1 - x^n \le n(1 - x)$ for all $x \in (0,1)$ and $n \ge 1$,
then we get the following upper bound for $\prob({\sf Q}_n^{(\ge 2)})$ continued from \eqref{eq_one_subtree_binomial_upper_bound} 
\aln{
&  \exptn \Big( Z_{n - \kappa} \mbbo \big( Z_{n - \kappa} > {\rm K} \mu^{n - \kappa} \big) \prob \big( \ol{\bf N}_{n, \kappa}^{(1)}(\varphi/2, \infty) \ge 1 \big) \Big[ 1 - \Big( \prob \big( \ol{\bf N}_{n, \kappa }^{(1)}(\varphi/2, \infty) = 0 \big) \Big)^{Z_{n - \kappa} - 1} \Big] \Big) \nonumber \\
& \hspace{.5cm} + \exptn \Big( Z_{n - \kappa} \mbbo( Z_{n - \kappa} \le {\rm K} \mu^{n - \kappa}) \prob \big( \ol{\bf N}_{n, \kappa}^{(1)}(\varphi/2, \infty) \ge 1 \big) \Big[ 1 - \Big( \prob \big( \ol{\bf N}_{n, \kappa}^{(1)}(\varphi/2, \infty) = 0 \big) \Big)^{Z_{n - \kappa} -1} \Big] \Big) \nonumber \\ 
& \le  \mu^{2n - 2 \kappa} \exptn \big[ W_{n - \kappa}^2 \mbbo \big( W_{n - \kappa} \le {\rm K} \big) \big] \Big[ \prob \big( \ol{\bf N}_{n, \kappa}^{(1)}(\varphi/2, \infty) \ge 1 \big) \Big]^2  \nonumber \\
& \hspace{2cm} + \mu^{n - \kappa} \exptn \big[ W_{n - \kappa} \mbbo \big( W_{n - \kappa} > {\rm K} \big) \big] \prob \big( \ol{\bf N}_{n, \kappa}^{(1)}(\varphi/2, \infty) \ge 1 \big)   =: {\cal I}_{n}^{(1)} + {\cal I}_n^{(2)}
}
where $W_{n - \kappa} = Z_{n - \kappa}/ \mu^{n - \kappa}$. The proof of the lemma would be complete if we establish
\aln{
& \limsup_{n \to \infty} r_n {\cal I}_n^{(1)} = 0 \label{eq_one_subtree_first_term_negligible}\\
& \mbox{and }  \lim_{{\rm K} \to \infty} \limsup_{n \to \infty} r_n {\cal I}_n^{(2)} = 0. \label{eq_one_subtree_second_term_negligible}
}
 

\noindent{\bf Proof of \eqref{eq_one_subtree_first_term_negligible}.} We shall use the union bound combined with the fact that the displacements are independent of the branching mechanism  to see that $\prob \big( \ol{\bf N}_{n, \kappa}^{(1)} (\varphi/2, \infty) \ge 1 \big) \le \mu^\kappa (\mu - 1)^\inv \prob \big( |X_1| > \gamma_n \varphi/2  \big)$. This observation yields the following upper bound for $r_n {\cal I}_n^{(1)}$ 
\aln{
\Big[ \mu^n \prob(|X_1| > \gamma_n) \Big] ~~\big[ {\rm K}^2 (\mu -1)^{- 2} \big] ~~ \Big[ \frac{ \prob(|X_1| > \gamma_n \varphi/2)}{ \prob(|X_1| > \gamma_n)} \Big]^2. 
}
The second term in the upper bound is finite and the ratio in the third term converges to $(\varphi/2)^{- \alpha}$ due to assumption~\eqref{ass:marginal:identical:dist}. The first term converges to zero by the choice of $\gamma_n$ in \eqref{eq:defn:gamman}.  Therefore, the proof is complete.


\noindent{\bf Proof of \eqref{eq_one_subtree_second_term_negligible}.} The union bound for $\prob \big( \ol{\bf N}_{n, \kappa}^{(1)}(\varphi/2, \infty) \ge 1 \big)$ to get following upper bound for $r_n {\cal I}_n^{(2)}$
\aln{
(\mu - 1)^\inv ~~ \big( \sup_{k \ge 1} \exptn[W_k \mbbo(W_k > {\rm K})] \big) ~~ \Big( \frac{\prob \big( |X_1| > \gamma_n \varphi/2 \big)}{ \prob \big( |X_1| > \gamma_n   \big)} \Big).
}
The ratio of the probabilities converges to $(\varphi/2)^{- \alpha}$ as $n \to \infty$. The second term converges to zero if we let ${\rm K} \to \infty$ as $(W_k : k \ge 1)$ is a sequence of uniformly integrable random variables under Kesten-Stigum condition.

\subsection{Proof of Proposition~\ref{lemma:one:subtree}}

It is immediate that $\hlsfun(\wtbfN_{n, \kappa}) = 0$ almost surely on the event ${\sf Q}_n^{(0)}$. This observation yields $\exptn^* \big( \hlsfun(\wtbfN_{n, \kappa}) \big) = \exptn^* \big[ \hlsfun(\wtbfN_{n, \kappa}) \mbbo({\sf Q}_n^{(1)}) \big] + \exptn^* \big[ \hlsfun(\wtbfN_{n, \kappa}) \mbbo({ \sf Q}_n^{(\ge 2)}) \big]$.  The proposition follows from the following two claims
\aln{
& \lim_{n \to \infty} \frac{r_n \exptn^*\big[ \hlsfun(\wtbfN_{n, \kappa}) \mbbo( {\sf Q}_n^{(1)}) \big]}{  \exptn \big[ \hlsfun \big( ~ \wtbfN_{n, \kappa}^{(1)} ~ \big) \big] / \prob(|X_1|  > \gamma_n )} = ( 1- p_e)^\inv \mu^{ - \kappa} \label{eq_propn_one_subtree_first_aim} \\
& \mbox{and } \limsup_{n \to \infty} r_n \exptn^* \big[ \hlsfun(~\wtbfN_{n, \kappa}~) \mbbo({\sf Q}_n^{(\ge 2)}) \big] = 0. \label{eq_propn_one_subtree_second_aim}
}
Note that  \eqref{eq_propn_one_subtree_second_aim} follows from Lemma~\ref{lemma:one:tree:large:disp} realizing $\hlsfun(~\wtbfN_{n, \kappa}~) \le 1$ almost surely. The claim in \eqref{eq_propn_one_subtree_first_aim} follows from 
\aln{
& \limsup_{n \to \infty} r_n \Big|  \exptn^* \Big( \hlsfun \big[ \wtbfN_{n, \kappa} \big]  \mbbo({\sf Q}_n^{(1)})  \Big)   - \exptn \Big( \hlsfun(\wtbfN_{n, \kappa}) \mbbo \big[ Z_{n - \kappa} > 0 \big] \frac{\mbbo \big[ {\sf Q}_n^{(1)} \big]}{1 - p_e} \Big) \Big| = 0 \label{eq:onesubtreesuffices:aim:step1}\\
    & \mbox{ and } \lim_{n \to \infty} \frac{ r_n  \exptn \Big[ \mbbo(Z_{n - \kappa} > 0) \hlsfun \big( \wtbfN_{n,\kappa} \big) \mbbo( {\sf Q}_n^{(1)})  \Big]}{ \mu^{- \kappa} \exptn \Big( \hlsfun \big(\wtbfN^{(1)}_{n, \kappa} \big) \Big) / \prob( |X_1 | > \gamma_n)} = 1. \label{eq:onesubtreesuffices:aim:step2}
}

\subsubsection{Proof of \eqref{eq:onesubtreesuffices:aim:step1}}

Define ${\cal S}_{n - \kappa}$ to be the event that at least one of the subtrees in the forest $({\sf T}_\kappa^{(j)} : 1 \le j \le Z_{n - \kappa})$ is infinite (never extincts). Realizing the fact that $\prob^*$ is the conditional version of $\prob$, we can connect them via Radon-Nikodym derivative which is as follows 
\alns{
 \dtv \prob^* = \frac{1}{\prob({\cal S})} \mbbo(Z_{n - \kappa} > 0) \mbbo({\cal S}_{n - \kappa}).   
}
This observation combined with the facts that $\prob({\cal S}) = ( 1 - p_e)$ and $\hlsfun(\wtbfN_{n, \kappa}^{(1)}) \le 1 $ almost surely, yields the following upper bound for the absolute difference of expectations in \eqref{eq:onesubtreesuffices:aim:step1} 
\aln{
( 1- p_e)^\inv r_n \exptn \Big[ \mbbo( Z_{n - \kappa}  > 0) \mbbo({\cal S}_{n - \kappa}^c) \mbbo({\sf Q}_n^{(1)}) \Big]. \label{eq_onesubtree_suffices_aim1_upp1}
}
Let ${\sf U}_j = \{ {\sf T}_\kappa^{(j)} \mbox{ becomes extinct } \}$. Then ${\cal S}_{n - \kappa}^c = \cap_{ j = 1}^{Z_{ n - \kappa}} {\sf U}_j $.  Note that ${\sf Q}_n^{(1)}$ and ${\cal S}_{n - \kappa}^c$  depends on $ Z_{n - \kappa}$. So, we shall first condition on $Z_{n - \kappa}$ and then compute an average with respect to the law of $Z_{n - \kappa}$. The union bound for the event ${\sf Q}_n^{(1)} = \cup_{j =1}^{Z_{n - \kappa}} \{ \ol{\bf N}_{n, \kappa}^{(j)} (\varphi/2, \infty) \ge 1 \}$ yields the following upper bound for the expression in \eqref{eq_onesubtree_suffices_aim1_upp1}
\aln{
& r_n \exptn \Big[ \mbbo(Z_{n - \kappa} > 0) \sum_{j =1}^{Z_{n - \kappa}} \exptn \Big( \mbbo(~ \ol{\bf N}_{n, \kappa}^{(j)} (\varphi/2, \infty)~ \ge 1)  \prod_{j'=1}^{Z_{n - \kappa}} \mbbo({\sf U}_{\rtext{j'}})  \Big| ~~Z_{n - \kappa} \Big) \Big] \nonumber \\
&  \le r_n\exptn \Big( \mbbo( Z_{n - \kappa} > 0)~~  Z_{n - \kappa} ~~p_e^{Z_{n - \kappa} } \Big) \prob \big( \  \ol{\bf N}_{n, \kappa}^{(1)} (\varphi/2, \infty) \ge 1 \big) \label{eq_onesubtree_free_star_upp}
} 
ignoring the constant $(1  - p_e)^\inv$ for obvious reasons. To obtain the last inequality, we have used the facts that $\ol{\bf N}_{n, \kappa}^{(j)}$ and $({\sf U}_{j'} : j' \in [1: Z_{n - \kappa}] \setminus \{j \})$ are independent and also identically distributed over $j \in [1 : Z_{n - \kappa}]$. The union bound yields $\prob(~ \ol{\bf N}_{n, \kappa}^{(1)} (\varphi/2, \infty) \ge 1~) \le (\mu - 1)^\inv \mu^{\kappa + 1} \prob(|X_1| > \gamma_n \varphi/2) $. This observation  leads to the following upper bound of \eqref{eq_onesubtree_free_star_upp}
\aln{
\Big[(\mu -1)^\inv \mu \Big] ~~\Big[ \frac{\prob(|X_1| > \gamma_n \rtext{\varphi/2})}{ \prob(|X_1| > \gamma_n)} \Big] \exptn\Big[ \mbbo(Z_{n - \kappa} > 0) W_{n - \kappa} p_e^{\rtext{Z_{n - \kappa}} } \Big]
}
where $W_k = \mu^{-k} Z_k$ for all $ k \ge 1$. The second term converges to $(\varphi/2)^{- \alpha}$. The term inside the expectation is bounded almost surely by $\sup_{k \ge 1} W_k$ which has finite expectation due to Kesten-Stigum condition \eqref{eq:assumption:branching:process}. So we are allowed to apply dominated convergence theorem if the almost sure limit is determined. It follows that $\mbbo(Z_{n - \kappa} > 0 ) \asconv \mbbo({\cal S})$ and $Z_{n - \kappa} \asconv \infty$ on the 
event ${\cal S}$ which imply that $p_e^{Z_{n - \kappa}} \asconv 0$ on ${\cal S}$. Therefore, the proof follows.  

\subsubsection{Proof of \eqref{eq:onesubtreesuffices:aim:step2}}

Note that the events $({\sf H}_n^{(j)} \cap \{ \cap_{t =1, t \neq j} ({\sf H}_n^{(t)})^c\} : j \in [1~:~ Z_{n - \kappa}])$ are mutually disjoint events. Recall that $\wtbfN_{n, \kappa}(g_i) = \sum_{j =1}^{Z_{n - \kappa}} \wtbfN_{n, \kappa}^{(j)}(g_i)$ almost surely and $(\varphi, \infty) \subset {\rm support}(g_1) \cap {\rm support}(g_2)$. Then it follows that $\hlsfun(\wtbfN_{n, \kappa}) = \hlsfun(\wtbfN_{n, \kappa}^{(j)})$ on the event ${\sf H}_n^{(j)} \{ \cap_{t = 1, t \neq j}^{Z_{n - \kappa}} ({\sf H}_n^{(t)})^c \}$. These facts together yield following expression for $\exptn \Big[ \hlsfun(\wtbfN_{n, \kappa}) \mbbo(Z_{n - \kappa} > 0)  \mbbo( {\sf Q}_n^{(1)})\Big]$
\aln{
& \exptn \Big[ \mbbo(Z_{n - \kappa} > 0) \sum_{j =1}^{Z_{n - \kappa}} \exptn \Big( \hlsfun \big[ \wtbfN_{n, \kappa}^{(j)} \big] \mbbo({\sf H}_n^{(j)})~\prod_{t =1, t \neq j}^{Z_{n - \kappa}} \mbbo \big[ ({\sf H}_n^{(t)})^c \big] ~\Big| Z_{ n- \kappa}  \Big)  \Big] \nonumber \\
& = \exptn \Big[ \mbbo(Z_{n - \kappa} > 0) ~Z_{n - \kappa}~ \exptn \Big( \hlsfun(\wtbfN_{n, \kappa}^{(1)}) \prod_{t =2}^{Z_{n - \kappa}} \mbbo( ({\sf H}_n^{(t)})^c) \Big| ~Z_{n - \kappa} \Big) \Big] \label{eq_one_subtree_suffices_aim_2_equality1}
}
as $\hlsfun(\wtbfN_{n, \kappa}^{(j)}) \mbbo({\sf H}_n^{(j)}) =  \hlsfun(\wtbfN_{n, \kappa}^{(j)})$ almost surely and $\big( \hlsfun(\wtbfN_{n, \kappa}^{(j)}) \prod_{t = 1, t \neq j}^{Z_{n - \kappa}} \mbbo( [{\sf H}_n^{(t)}]^c) : j \in [1~:~ Z_{n - \kappa}~] \big)$ are i.i.d. conditioned on $Z_{n - \kappa}$. Our next observation is that $\wtbfN_{n, \kappa}^{(1)}$ and $\prod_{j =2}^{Z_{n - \kappa}} \mbbo([{\sf H}_n^{(t)}]^c)$ are independent conditioned on $Z_{n - \kappa}$. These observations suggest the following expression for the right hand side of \eqref{eq_one_subtree_suffices_aim_2_equality1}
\aln{
\exptn \Big[ \hlsfun \big[ \wtbfN_{n, \kappa}^{(1)} \big] \Big] \exptn \Big[ \mbbo \big( Z_{n - \kappa} > 0 \big) Z_{n - \kappa} ~~\big[ \prob \big( [{\sf H}_n^{(1)}]^c \big) \big]^{Z_{n - \kappa} - 1} \Big] \label{eq_one_subtree_sffices_aim2_equality2}
}
as $( \mbbo( {\sf H}_n^{(t)}) : t \in [1~:~|{\sf D}_{n - \kappa}|])$ are i.i.d. conditioned on $|{\sf D}_{n - \kappa}|$ and ${\sf H}_n^{(t)}$ is independent of ${\sf D}_{n - \kappa}$ for every fixed $t$. In the light of the expression \eqref{eq_one_subtree_sffices_aim2_equality2} obtained for the expectation in the numerator of \eqref{eq:onesubtreesuffices:aim:step2}, the claim in \eqref{eq:onesubtreesuffices:aim:step2} reduces to 
\aln{
\lim_{n \to \infty} \exptn \Big[ \mbbo(Z_{n - \kappa} > 0) W_{n - \kappa} \big(1 - \prob[ {\sf H}_n^{(1)}] \big)^{Z_{n - \kappa} - 1} \Big] = 1.  \label{eq_one_sutree_suffices_step_2_final aim}
}

\begin{lemma} \label{lemma_onesubtreesuffices_second_aim}
If the assumptions of Proposition~\ref{lemma:one:subtree} hold, then $[1 - \prob({\sf H}_n^{(1)})]^{Z_{n - \kappa}} \asconv 1$.
\end{lemma}

 In view of the above lemma, we can see that the random variables inside expectation in \eqref{eq_one_sutree_suffices_step_2_final aim} converges almost surely to $\mbbo({\cal S}) W$ and $\exptn(\mbbo({\cal S}) W) = \exptn(W) = 1$. The dominated convergence theorem allows the exchange of limit and expectation  as $\mbbo(Z_{n - \kappa} > 0) W_{n - \kappa} \big(1 - \prob[ {\sf H}_n^{(1)}] \big)^{Z_{n - \kappa} - 1} \le \sup_{k \ge 1} W_k$ which has finite expectation (thanks to Kesten-Stigum condition \eqref{eq:assumption:branching:process}).  Thus, the proof of \eqref{eq:onesubtreesuffices:aim:step2}  is complete. 
 
\medskip 

\noindent{\bf Proof of Lemma~\ref{lemma_onesubtreesuffices_second_aim}.} To prove the lemma, it is enough to show that 
\aln{
\limsup_{n \to \infty} \Big[ 1 - \Big(1 - \prob({\sf H}_n^{(1)}) \Big)^{Z_{n - \kappa}} \Big] = 0 \label{eq_lemma_aux_prop_onesubtree_aim}
}
almost surely. If we use the inequality $ 1 - x^n \le n( 1- x)$ for all $n \ge 1$ and $x \in (0,1)$, then it follows that the left hand side  of \eqref{eq_lemma_aux_prop_onesubtree_aim} can be almost surely bounded from above by $Z_{n - \kappa} \prob( \ol{\bf N}_{n, \kappa}(\varphi/2, \infty) \ge 1)$. The union bound can be used to bound the probability $\prob( \ol{\bf N}_{n, \kappa}(\varphi/2, \infty) \ge 1)$ and it leads to the following almost sure upper bound for the left hand side of \eqref{eq_lemma_aux_prop_onesubtree_aim}  
\aln{
\Big[ \mu^n \prob(|X_1| > \gamma_n) \Big]~~ W_{n - \kappa}~~\Big[ \frac{ \prob(|X_1|  > \gamma_n \varphi/2)}{ \prob( |X_1| > \gamma_n )} \Big] (\mu/(\mu - 1)).
}
In the above-mentioned upper bound, the first term converges to zero due to the choice of $\gamma_n$ given in  \eqref{eq:defn:gamman}. The other terms are almost surely bounded and so, the proof is complete.

\subsection{Proof of Proposition~\ref{propn_pruned_subtree}}

Using the inequality in \eqref{eq_comp_ineq_hlsfun_two_measures}, the proposition reduces to the following claim
\aln{
\lim_{\varrho \to \infty} \limsup_{n \to \infty}~~~ [\prob(|X_1| > \gamma_n)]^\inv \exptn \Big| \wtbfN_{n, \kappa}^{(1)}(g) - \wtbfN_{n, \kappa, \varrho}^{(1)}(g) \Big| = 0 \mbox{ for all } g \in C_c^+(\mbbr_0). \label{eq_pruning_reduced_aim}
}
To prove this claim, we shall use assistance of another intermediate point process and we shall construct this here using the following rewarding scheme.

\subsubsection{Contruction of the marked subtree.}

We start with the root of the subtree $\sfT_{\kappa}^{(1)}$ and reward it with zero $\diamond$. Let us consider a particle in the $i$-th generation which is rewarded with $j$ $\diamond$s. Its first $\varrho$ children (according to the lexicographic order of their Ulam-Harris labels) are awarded with the $(j+1)$ $\diamond$s. If the particle has $l (> \varrho)$ children, then we assign to each of the left $(l - \varrho)$ child/children with $j$ $\diamond$s as reward. This scheme will be followed until each particle in the $\kappa$-th generation is rewarded. The resultant subtree will be referred to as the {\it marked subtree} and will be denoted by $\sfT_\kappa^{(\diamond, 1)}$. Note that the vertices in $\sfT_\kappa^{(1)}$ is same as that of $\sfT_{\kappa}^{(\diamond,1)}$ (ignoring their rewards) and the set of all vertices in the $i$-th generation of $\sfT_\kappa^{(\diamond, 1)}$ are denoted by ${\sf D}_i^{(1)}$. For every $\sfu \in {\sf D}_i^{(1)}$ rewarded with $l$ $\diamond$s, we define $A^{(\varrho, \diamond)}(\sfu)$ to be the number of descendants in the $\kappa$-th generation of $\sfu$ which are rewarded with $(\kappa - i + l)$ $\diamond$s for $l \in [0~:~i]$. The following is the most important observation for us:
\aln{
 & \mbox{ If $\sfu \in {\sf D}_i^{(1)}$ is rewarded with $i$ $\diamond$s, then $\sfu \in \mathbb{T}_\kappa^{(1)}$ (pruned version of $\sfT_\kappa^{(1)}$)}  \nn \\
 & \mbox{ and it has $A^{(\varrho, \diamond)}(\sfu)$ descendants in the $\kappa$-th generation of $\mathbb{ T}_\kappa^{(1)}$.}
 }
 Define 
\aln{
{\bf N}_{n, \kappa, \varrho}^{(*, 1)} = \sum_{\sfu \in \sfT_\kappa^{(*,1)} \setminus \{{\sf r}_1\}}A^{(\varrho, \diamond)}(\sfu) \bdelta_{\gamma_n^\inv X(\sfu)}. 
}    
Then the claim in \eqref{eq_pruning_reduced_aim} follows if, for every $g \in \ccr$,  we can show
\aln{
& \lim_{ \varrho \to \infty} \limsup_{n \to \infty} ~~[\prob(|X_1| > \gamma_n)]^\inv \exptn  \Big| \wtbfN_{n, \kappa}^{(1)}(g) - {\bf N}_{n, \kappa, \varrho}^{(*,1)}(g) \Big| = 0  \label{eq_comp_marked_preprun_aim}\\
& \mbox{ and } \lim_{\varrho \to \infty} \limsup_{n \to \infty} ~~[\prob(|X_1| > \gamma_n)]^\inv \exptn \Big| {\bf N}_{n, \kappa, \varrho}^{(*, 1)}(g)  -  \wtbfN_{n, \kappa, \varrho}^{(1)}(g)\Big| = 0. \label{eq_comp_marked_pruned_aim}
}

\subsubsection{Proof of \eqref{eq_comp_marked_preprun_aim}}

Let $\varphi_0 > 0$ be such that ${\rm support}(g) \subset \{ x : |x| >  \varphi_0\}$ and define $\norm{g} = \sup_{x} |g(x)|$. Let $(Z_i^{(\varrho)} : i \ge 0)$ denote the sizes of the generations of the branching process with progeny distribution $Z_1^{(\varrho)} \eqd  Z_1 \wedge \varrho$.  Note that $(A^{(\varrho, \diamond)} (\sfu) : \sfu \in {\sf D}_{i}^{(1)})$ are independent copies of $ Z^{(\varrho)}_{\kappa - i}$ for every $i \ge 1$. As $A^{(\varrho, \diamond)}(\sfu) \le A(\sfu)$ almost surely for every $\sfu \in \sfT_\kappa^{(1)}$, we have following upper bond for the expectation of the absolute difference in \eqref{eq_comp_marked_preprun_aim} 
\aln{
& \norm{g} \sum_{i =1}^\kappa \exptn \big[ \sum_{\sfu \in {\sf D}_i^{(1)}} [A(\sfu) - A^{(\varrho, \diamond)}(\sfu)] \mbbo(|X(\sfu)| > \gamma_n \varphi_0)\big] \nonumber \\
& = \norm{g} \prob \big( |X_1| > \gamma_n \varphi_0 \big) \sum_{i =1}^\kappa \exptn \Big( |{\sf D}_i^{(1)}| \Big) \exptn \big( Z_{\kappa - i} - Z^{(\varrho)}_{\kappa - i} \big).  \label{eq_upp_abs_diff_marked_prepruned}
} 
For the last equality,  we have used the independence of the branching mechanism and the identically distributed displacements. Additionally, we used the subtrees of $\sfT_\kappa^{(1)}$ which are rooted at ${\sf D}_i^{(1)}$ are i.i.d. and independent of ${\sf D}_i^{(1)}$. If we use this observation in \eqref{eq_upp_abs_diff_marked_prepruned}, we obtain following upper bound for the left hand side of \eqref{eq_comp_marked_preprun_aim}
\aln{
\norm{g}\frac{\prob(|X_1| > \gamma_n \varphi_0)}{\prob(|X_1| > \gamma_n )} \sum_{i =1}^\kappa \mu^i \big( \exptn[Z_{\kappa - i}] - \exptn[Z_{\kappa - i}^{(\varrho)}] \big). 
}
The ratio of the probabilities converge to $\varphi_0^{- \alpha}$ due to Assumptions~1.2 on the displacements as $n \to \infty$ and the other term do not depend on $n$. As the sum is finite, it vanishes when we let $\varrho \to \infty$.

\subsubsection{Proof of \eqref{eq_comp_marked_pruned_aim}}
We can use the same facts used in the proof of \eqref{eq_comp_marked_preprun_aim} with appropriate modifications, to derive the following upper bound for the left hand hand side of \eqref{eq_comp_marked_pruned_aim}
\aln{
\norm{g} \frac{\prob(|X_1| > \gamma_n \varphi_0)}{\prob(|X_1| > \gamma_n  )} \sum_{i =1}^\kappa \exptn(Z_{\kappa - i}^{(\varrho)}) \big( \exptn(Z_i) - \exptn(Z_i^{(\varrho)}) \big).
}
We have already seen that the ratio of the probabilities converge to $\varphi_0^{- \alpha}$. The sum again vanishes if we let $\varrho \to \infty$ due to having finitely many summands. We conclude the proof here.

\subsection{Proof of Proposition~\ref{lemma:computing:limit}}


In this proof, we shall use the \rtext{couple} $(l, t)$ to  denote the $t$-th vertex $\sfu$ at the $l$-th generation of the regularized subtree $\scrt_\kappa^{(1)}$. We did not reserve any duple for the root ${\varpi}_1$ of $\scrt_\kappa^{(1)}$ as its displacement is assumed to be zero.   Define the random vectors
\aln{
\widetilde{\bf A} = (A^{(\varrho)}{(l,t)} : 1 \le t \le \varrho^l, 1 \le l \le \kappa) \mbox{ and } \widetilde{\bf X} = (X'{(l,t)} : 1 \le t \le \varrho^l, 1 \le l \le \kappa).
} 
Let $[0~:~i] = \{0,1,2, \ldots, i\}$ for every $i \ge 1$. Then the state spaces of $\wt{\bf A}$ and $\wt{\bf X}$ will denoted respectively by 
$${\sf B}_\varrho = [0: ~\varrho^{\kappa - 1}]^\varrho \times [0: ~\varrho^{\kappa - 2}]^{\varrho^2} \times \ldots \times [0:~\varrho]^{\varrho^{\kappa - 1}} \times [0: ~1]^{\varrho^\kappa}$$
  and $ {\sf R}_\varrho =  \mbbr^{\varrho + \varrho^2 + \varrho^3 + \ldots + \varrho^\kappa}$
 Let $\wt{\bf x} = (x_{l,t} : 1 \le t \le \varrho^l; ~1 \le l \le \kappa)$.  These notations yield following expression for the  left hand side of \eqref{eq:propn:comp:lim} 
\aln{
\mu^{- \kappa} \sum_{\wt{\bf a} \in {\sf B}_\varrho} \prob(\wt{\bf A} = \wt{\bf a}) \int_{{\sf R}_\varrho} \prod_{i =1}^2 \Big[ 1 - \exp \Big\{ - \Big( \sum_{l =1}^\kappa \sum_{t =1}^{\varrho^l} a_{l,t} g_i \big( x_{l,t} \big) - \epsilon_i \Big)_+ \Big\} \Big] \frac{\prob \Big( \gamma_n^\inv \wt{\bf X} \in \dtv \wt{\bf x} \Big)}{\prob(|X_1| > \gamma_n)} \label{eq:aim:limpropn:new:notation}
}
and we shall compute the limit as $n \to \infty$, $\varrho \to \infty$ and $\kappa \to \infty$ respectively in the rest of the proof.

\noindent{\bf Limit of \eqref{eq:aim:limpropn:new:notation} as $n \to \infty$. } Note that only the law of $\wt{\bf X}$ involves $n$. We first divide the members of $\wt{\bf X}$ into groups such that the members of a group have the same parent and members of two different groups have different parent.   To be precise, the groups are $(X'(l,t) : t \in {\sf J}_l)$ where ${\sf J}_l = \{\varrho k + 1 : k \in [0~:~ \varrho^{l -1} - 1]\} $ for every $l \in [1~:~ \kappa]$.  It is clear that the groups are i.i.d. and regularly varying.  To be more formal, let ${\rm PROJ}_\varrho : \mbbr^\mbbn \to \mbbr^\varrho$ such that ${\rm PROJ}_\varrho  ((x_i : i \ge 1)) = (x_i : 1 \le i \le \varrho)$.  \label{not_proj} It follows from the assumption~\ref{ass:joint:regvar:disp} and Theorem~4.1 in \cite{lindskog:resnick:roy:2014} that 
$$ \prob \Big[ \big((X'(l, t), X'(l, t + 1), \ldots X'(l, t + \varrho) \big) \in \cdot \Big] \in {\rm RV}_\alpha( \mbbr^\varrho \setminus \{ {\bf 0}_\varrho\},  \lambda_0^{(\varrho)})$$
 where $\lambda_0^{(\varrho)} = \lambda_0 \circ {\rm PROJ}_\varrho^\inv$ for every $t \in {\sf J}_l$ and $1 \le l \le \kappa$. This decomposition combined with a derivation similar to the Subsection~4.5.1 in \cite{lindskog:resnick:roy:2014} yields  $\prob(\wt{\bf X} \in \cdot) \in {\rm RV}_\alpha({\sf R}_\varrho \setminus \{ {\bf 0}_{\varrho + \varrho^2 + \ldots+ \varrho^{\kappa}}\}, \tau( \cdot ))$ where
\aln{
& \tau (\cdot)  : = \sum_{l =1}^\kappa \sum_{t \in {\sf J}_l} \tau_{l, t} (\cdot) \mbox{ such that } \tau_{l, t} (\cdot) = \bigotimes_{j =1}^{\varrho + \varrho^2 + \varrho^{l -1} + t - 1} \bdelta_0 \otimes \lambda_0^{(\varrho)} \bigotimes_{j'= \varrho + \varrho^2 + \ldots + t + \varrho}^{\varrho +\varrho^2 + \ldots + \varrho^\kappa} \bdelta_0 (\cdot).
}   
Recall the set ${\cal O} = \{{\bf x} \in \mbbr^\mbbn : |x_1| > 1 \}$ defined in \eqref{defn_dashed_lamda}. These facts combined with Theorem~3.1 in \cite{lindskog:resnick:roy:2014} (see proof of (ii) implies (iii) in the Definition~3.2) yield 
\aln{
& \frac{\prob( \gamma_n^\inv \wt{\bf X} \in \cdot)}{ \prob(|X_1| > \gamma_n)} \hlconv \tau'(\cdot)= \sum_{l =1}^\kappa \sum_{t \in {\sf J}_l} \tau'_{l, t} (\cdot) \mbox{ where }\\
&  \tau'_{l, t} (\cdot) = \bigotimes_{j =1}^{\varrho + \varrho^2 + \varrho^{l -1} + t - 1} \bdelta_0 \otimes \lambda^{(\varrho)} \bigotimes_{j'= \varrho + \varrho^2 + \ldots + t + \varrho}^{\varrho +\varrho^2 + \ldots + \varrho^\kappa} \bdelta_0 (\cdot) \mbox{ and } \lambda^{(\varrho)} (\cdot) = \lambda_0^{(\varrho)} (\cdot) / \lambda_0({\cal O }).
}
So we accomplished the first task by obtaining the limit of the measure inside the integral in \eqref{eq:aim:limpropn:new:notation}.  

Note that the limit can be pushed inside sum in  the expression \eqref{eq:aim:limpropn:new:notation} as the sum is finite and the law of $\wt{\bf A}$ does not depend on $n$. To push the limit inside the integral, we have to show that the integrand is a member of the class ${\cal C}_b({\sf R}_\varrho \setminus \{ {\bf 0}_{\varrho + \varrho^2 + \ldots + \varrho^{\kappa- 1}}\})$. Once limit goes inside inside the integral ($\tau'$ replaces the ratio of probabilities in \eqref{eq:aim:limpropn:new:notation}), we need to do some algebra in order to discount the vertices $(l,t)$ such that $A^{(\varrho)}(l,t) = 0$.  The details 
 are similar to the derivation from (4.19) till (4.24) in  \cite{bhattacharya:hazra:roy:2017} and so skipped here. To write down the limit explicitly, we need following notations.

\begin{itemize}

\item $U^{(\varrho)}$ is an independent copy of the random variable $Z_1^{(\varrho)}$. $\wt{U}^{(\varrho)}$ denotes the random variable $U^{(\varrho)}$ conditioned to be positive that is, $\prob(\wt{U}^{(\varrho)} = i) = \prob(U^{(\varrho)} = i | U^{(\varrho)} > 0)$ for every $i \ge 1$. 


\item $(Z_l^{(s, \varrho)}: l \ge 1, s \ge 1)$ is an collection of independent random variables such that $Z_l^{(s, \varrho)} \eqd Z_l^{(\varrho)}$ where $Z_l^{(\varrho)}$ is the size of the $l$-th generation of a GW process with progeny random variable $Z_1^{(\varrho)}$. The random variable $\wt{Z}_l^{(s, \varrho)}$ will stand for the random variable $Z_l^{(s, \varrho)}$ conditioned to be positive. We assume that the collection $(Z_l^{(s,\varrho)} : l \ge 1, s \ge 1)$ is independent of the random variable $U^{(\varrho)}$. 


\end{itemize}

The limit of the expression in \eqref{eq:aim:limpropn:new:notation} equals 
\aln{
 &\mu^{- \kappa} \int_{\scrm_0} \hlsfun(\nu) m_*^{(\kappa, \varrho)}(\dtv \nu) \label{eq_lim_n_final}\\
& \mbox{where }~ m_*^{(\kappa, \varrho)}(\cdot)  = \sum_{l =0}^{\kappa - 1} \mu_\varrho^{\kappa - l - 1} \prob(U^{(\varrho)} \ge 1) \exptn \Big[ \sum_{{\sf A} \in \pow\big( [1~:~\wt{U}^{(\varrho)}] \big) \setminus \{\emptyset \}  } \lambda^{(\varrho)} \big( {\bf x} \in \mbbr^\varrho : \sum_{s \in {\sf A}} \wt{Z}_l^{(s, \varrho)} \bdelta_{x_s} \in \cdot \big) \nn \\
& \hspace{2.5cm} \big[ \prob(Z_l^{(\varrho)} \ge 1) \big]^{|{\sf A}|} \big[ \prob (Z_l^{(\varrho)} = 0) \big]^{\wt{U}^{(\varrho)} - |{\sf A}|}   \Big].
}
 Recall ${\sf B}_r(\xi_0) = \{ \nu \in \scrm_0 :  \rho_v(\nu, \varnothing) > r \}$ where $ \varnothing $ is the null measure on $\mbbr_0$.  To claim 
\aln{
\prob(\wtbfN_{n, \kappa, \varrho}^{(1)} \in \cdot) / \prob(|X_1| > \gamma_n) \hlconv m_*^{(\kappa, \varrho)} (\cdot) \label{eq_lim_n_hlconv_final}
}
 in the space of measures on $\scrm_0$, we must check whether $m_*^{(\kappa, \varrho)}$ is a non-null measure on $\scrm_0$ and satisfies $ m_*^{(\kappa, \varrho)} ({\sf B}_r) < \infty$ for every $r > 0$ in addition to the convergence. 


\noindent{\bf Rest of the proof}. Observe that $\mu_\varrho \to \mu$, $U^{(\varrho)} \asconv U$, $\wt{U}^{(\varrho)} \asconv \wt{U}$ and $( \wt{Z}_l^{(s, \varrho)} : s \ge 1, l \ge 1) \asconv (\wt{Z}_l^{(s)} : s \ge 1, l \ge 1)$ which are introduced before the main result Theorem~\ref{thm:main_result}.  Using Theorem~4.1 (combined with the Remark on page~296) in \cite{lindskog:resnick:roy:2014} and dominated convergence theorem (combined with the above-stated almost sure convergences), we see that 
\aln{
& \mu^{-\kappa} \int \hlsfun(\nu) m_*^{(\kappa, \varrho)}(\dtv \nu) \rightarrow \int \hlsfun(\nu) m_*^{(\kappa)}(\dtv \nu) \\
& \mbox{where }  m_*^{(\kappa)} (\cdot) = \sum_{l = 0}^{\kappa - 1} \mu^{- l - 1} \prob(U \ge 1) \exptn \Big[ \sum_{{\sf A} \in {\rm Pow}([1~:~ \wt{U}])} \big[ \prob(Z_1 \ge 1) \big]^{|{\sf A}|} \big[ \prob(Z_1 = 0) \big]^{\wt{U} - |{\sf A}|}   \nn\\
& \hspace{2cm} \lambda(\{{\bf x} \in \mbbr^\mbbn : \sum_{s \in {\sf A}} \wt{Z}_l^{(s)} \bdelta_{x_s} \in \cdot \})  \Big]. \label{eq_defn_mstarkappa}
}
To complete the proof of the fact that $\mu^{-\kappa} m_*^{(\kappa, \varrho)} \hlconv m_*^{(\kappa)}$, we need to verify that $m_*^{(\kappa)}$ is a non-null measure on $\scrm_0$ and $ m_*^{(\kappa)}({\sf B}_r^c(\varnothing)) < \infty$. These technicalities are ignored as they are addressed in \cite{bhattacharya:hazra:roy:2017}. Note that $(1 - e^{- x}) \le 1$  for $x \ge 0$ and $\sum_{l \ge 1} \mu^{- l} < \infty$ as $\mu >1$.  This observation allows us to apply the dominated convergence theorem and conclude $\int \hlsfun(\nu) m_*^{(\kappa)}(\dtv \nu) \rightarrow ( 1- p_e)\int \hlsfun(\nu) m_*(\dtv \nu)$ as $\kappa \to \infty$ where $m_*$ is defined in \eqref{eq:derived_lim_measure_K}.  Ignoring again the additional technical details, we conclude $m_*^{(\kappa)} \hlconv m_*$ and the proof of the proposition follows.

\section{Proofs of the other results in Section~1 and Section~2}

\subsection{Proof Theorem~\ref{thm:ldp:rightmost:joint:regvar}}

Recall that  $\{M_n > \gamma_n x\} = \{{\bf N}_n \in {\sf E}_x \}$ where ${\sf E}_x = \{\xi \in \scrm_0: \xi (x, \infty) \ge 1 \}$ for all $x > 0$ and the form of $m^*$ given in \eqref{eq:derived_lim_measure_K}.   Note that each member of$(\wt{Z}_l : l \ge 0)$ is supported on the set of all positive integers.  Ignoring the abstract technicalities,  some algebra leads to the following identity
\aln{
& \lim_{n \to \infty} r_n \prob^*(M_n > \gamma_n x) \nonumber \\
& = (1 - p_e)^\inv \prob(U > 0) \sum_{l=0}^\infty \mu^{-(l+1)} \exptn \Big[ \sum_{{\sf G} \in \pow([1~:~\wt{U}])\setminus \{ \emptyset \} }  \Big( \prob(Z_l > 0) \Big)^{|{\sf G}|} \Big( \prob(Z_l = 0) \Big)^{\wt{U} - |{\sf G}|} \nonumber \\
& \h \h \lambda \Big(\{{\bf y} \in \mbbr^\mbbn : \sum_{s \in {\sf G}} \bdelta_{y_s} (x, \infty) \ge 1\} \Big) \Big].  \label{eq:jtregvar:maxima:one}
}
We now use homogeneity property of the limit measure $\lambda_0$ to see $\lambda \big( \big\{ {\bf y} \in \mbbr^\mbbn : \sum_{s \in {\sf G}} \bdelta_{y_s} (x, \infty) \ge 1 \big\} \big) = x^{- \alpha} \lambda \big(  \big\{ {\bf y} \in \mbbr^\mbbn: \sum_{s \in {\sf G}} \bdelta_{y_s}(1, \infty) \ge 1  \big\} \big)$. It is \rtext{straightforward} to check that $\{ {\bf y} \in \mbbr^\mbbn: \sum_{s \in {\sf G}} \bdelta_{y_s}(1, \infty) \ge 1\} = \bigcup_{s \in {\sf G}} {\sf V}_s$ where ${\sf V}_s$ is introduced in Theorem~\ref{thm:ldp:rightmost:joint:regvar}.  Hence the expression in \eqref{eq:maxima:jtregvar} follows.

It follows from Corollary~5.1 in \cite{hult:samorodnitsky:2010} that $\varnothing \notin {\rm cl}({\sf E}_x)$. To conclude the proof, it is then enough to show that $m^*(\partial {\sf E}_x) = 0$ which we shall establish now. Note that ${\rm cl}({\sf E}_x) \subset  \ol{\sf E}_x :=  \{ \xi \in \scrm_0 : \xi [x, \infty) \ge 1 \}$. So it is clear that $m^*(\partial {\sf E}_x) \le  m^*(\ol{\sf E}_x \setminus {\sf E}_x) = m^*(\ol{\sf E}_x) - m^*({\sf E}_x)$. Define $\wt{\sf V}_s = \pi_s^\inv[1, \infty)$. We can now use the explicit form of the $m^*$ in \eqref{eq:derived_lim_measure_K} to obtain following upper bound for the difference
\aln{
& {\rm C}_4 \sum_{l =0}^\infty  \mu^{-(l+1)} \exptn \Big[ \sum_{{\sf G} \in \pow([1~:~\wt{U}]) \setminus \{ \emptyset \} } \lambda\big( \big\{\cup_{s \in {\sf G}} \wt{\sf V}_s \big\} \setminus \big\{ \cup_{s \in {\sf G}} {\sf V}_s \big\} \big)   \big( \prob(Z_l = 0) \big)^{\wt{U}} \Big( \frac{ \prob(Z_l > 0) }{ \prob(Z_l = 0) } \Big)^{|{\sf G}|} \Big] \label{eq:limit:bounadry:Bx}
}
We now would like to derive an upper bound for $\lambda\big( \big\{\cup_{s \in {\sf G}} \wt{\sf V}_s \big\} \setminus \big\{ \cup_{s \in {\sf G}} {\sf V}_s \big\} \big)$. Let ${\sf G}$ be a singleton set, then $\{\pi_s^\inv [1, \infty) \} \setminus \{ \pi_s^\inv (1, \infty)\} = \pi_s^\inv \{1\}$. Suppose now that ${\sf G} = \{s_1, s_2 , \ldots, s_k\}$ for $k > 1$. For every $k \ge 1$, $\pi_{s_1, s_2, \ldots, s_k} \mbbr^\mbbn \to \mbbr^k$ will be used to denote a projection map such that $\pi_{s_1, s_2, \ldots, s_k}((x_i : i \ge 1)) = (x_{s_1}, x_{s_2}, \ldots, x_{s_k})$. \label{nota_projection_map_general} Then, $\Big\{ \bigcup_{i =1}^k \pi_{s_i}^\inv[1, \infty) \Big\} \setminus \Big\{ \bigcup_{i =1}^k \pi_{s_i}^\inv (1, \infty) \Big\}$ equals
\alns{
 \bigcup_{i =1}^k \Big\{ \pi^\inv_{s_1, s_2 , \ldots, s_k} \Big( \otimes_{j=1}^{i -1} (- \infty, 1] \times \{ 1 \} \otimes_{j'=i +1}^k (-\infty, 1] \Big) \Big\} \subset \bigcup_{i =1}^k \pi_{s_i}^\inv (\{1 \}).
}
This observation yields following upper bound for the expression derived in \eqref{eq:limit:bounadry:Bx}
\aln{
{\rm C}_4 \sum_{l = 0}^\infty \mu^{- (l+1)} \exptn \Big[ \sum_{{\sf G} \in \pow([1~:~\wt{U}]) \setminus \{\emptyset\}} \lambda \big( \bigcup_{s \in G} {\pi_s^\inv \{1\}}\big) \big[ \prob(Z_l > 0) \big]^{|{\sf G}|} \big[ \prob(Z_l = 0) \big]^{\wt{U } - |{\sf G}|} \Big]. \label{eq:limit:measure:maxima:Bx}
}
Then, we can use union bound to get that $\lambda(\bigcup_{s \in |G|} \pi_s^\inv\{1\}) \le \sum_{s \in {\sf G}} \lambda \circ \pi_s^\inv( \{ 1\}) = |{\sf G}| \nu_\alpha(\{1\}) = 0$ as the displacements are identically distributed and $\nu_\alpha$ is absolutely continuous with respect to the Lebesgue measure. This completes the proof.

\subsection{Proof of Theorem~\ref{thm:extreme:positions:linear:dependent:disp}.}

We divide the proof into two parts. In the first part, we establish that the random vector $\prob({\bf X} \in \cdot) \in {\rm RV}_\alpha((0, \infty)^2, \lambda_{\phi})$. In the next part, we shall derive the expression for $m_\phi^*$ from $\lambda_\phi$ using Theorem~\ref{thm:main_result}. 

\subsubsection{Derivation of $\lambda_\phi$ and proof of \eqref{eq:expression:limit:measure:linearly:dependent:disp}}
Consider the map ${\mathscr T}_\phi : (0, \infty)^3 \to (0, \infty)^2$ such that ${\mathscr T}_\phi(x_1, x_2, x_3) = (\phi x_1 + x_2, \phi x_2 + x_3 )$. It follows immediately that ${\mathscr T}_\phi$ is continuous and ${\mathscr T}_\phi({\bf 0}_3) = {\bf 0}_2$. \rtext{Note} that $\prob({\bf Y} \in \cdot) \in {\rm RV}_\alpha((0, \infty)^3, \lambda^{({\bf Y})})$ where 
\aln{
\lambda^{({\bf Y})} = \nu_\alpha^+ \otimes \bdelta_0 \otimes \bdelta_0 + \bdelta_0 \otimes \nu_\alpha^+ \otimes \bdelta_0 + \bdelta_0 \otimes \bdelta_0 \otimes \nu_\alpha^+.
}
Therefore, for every ${\sf A} \in {\cal B}((0, \infty)^2)$ such that $ \dist({\bf 0}_2, {\sf A}) > 0$, we can use Theorem~2.2 in \cite{lindskog:resnick:roy:2014} to derive 
\alns{
& \lambda_\phi \big( {\sf A} \big) = \lambda^{({\bf Y})} \big( \big\{ {\bf z} \in (0, \infty)^3 :{\mathscr T}_\phi({\bf z}) \in {\sf A}  \big\} \big) \nonumber \\
&=\phi^\alpha \nu_\alpha^+(\{z > 0 : z \cdot (1,0) \in {\sf A}\}) + \nu_\alpha^+(\{z > 0 : z\cdot(1, \phi) \in {\sf A}\}) + \nu_\alpha^+(\{ z > 0 : z\cdot (0,1) \in {\sf A}\}) .
}
The homogeneity property of the measure $\nu_\alpha^+$ has been used to obtain the last equality. The proof of \eqref{eq:expression:limit:measure:linearly:dependent:disp} follows from Theorem~\ref{thm:main_result} replacing $\lambda$ by $\lambda_\phi$.

\subsubsection{Proof of \eqref{eq_two_largest_order_stat_lin_dep} and \eqref{eq_max_lin_dep}}
Let ${\sf E}_{x_1, x_2} = \{\xi \in \scrm_0 : \xi(x_1, \infty) \ge 1 \mbox{ and } \xi(x_2, \infty) \ge 1\}$ for $0 < x_2 < x_1 < \infty$. Then it follows that 
\alns{
\big\{ {\bf N}_n \in {\sf E}_{x_1, x_2} \big\} = \big\{ M_n^{(1)} > {\gamma}^{(\phi)}_n x_1 \big\} \cap \big\{ M_n^{(2)} > {\gamma}^{(\phi)}_n x_2 \big\} 
}
If we ignore the technicalities by assuming 
\aln{
 \varnothing \notin {\rm cl}({\sf E}_{x_1, x_2})  \mbox{ and } m_\phi^*(\partial {\sf E}_{x_1, x_2}) = 0, \label{eq_technical_cond_lin_dep}
}
then Theorem~\ref{thm:main_result} gives us 
\aln{
& \lim_{n \to \infty} r_n^{(\phi)} \prob \big( M_n^{(1)} >\gamma_n^{(\phi)} x_1;~~ M_n^{(2)} > \gamma^{(\phi)}_2 x_2 \big) \nonumber \\
& = ( 1 + \phi^{\alpha}) \sum_{l = 0}^\infty 2^{- l} \nu_\alpha^+ \big( \big\{ z > 0 : 2^l \bdelta_z(x_1, \infty) \ge 1 \mbox{ and } 2^l \bdelta_z(x_2, \infty) \ge 2 \big\} \big) \nonumber \\
&  + \sum_{l =0}^\infty 2^{- l}\nu_\alpha^+ \big( \big\{z > 0 : 2^l (\bdelta_z + \bdelta_{\phi z})(x_1,\infty)\ge 1 \mbox{ and } 2^l(\bdelta_z + \bdelta_{\phi  z} ) (x_2, \infty) \ge 2 \big\} \big). \label{eq_first_expression_lim_measure_lin_dep}
}
We shall treat each of the terms in \eqref{eq_first_expression_lim_measure_lin_dep} separately. If we look at the first term, then the term $l = 0$ vanishes as $\bdelta_z (x_2,\infty) < 2$.  If $l \ge 1$, then a simple argument combined with the fact $\phi > 1$ shows that $\nu_\alpha^+( \{z > 0 : 2^l \bdelta_z(x_1, \infty) \ge 1 \mbox{ and } 2^l \bdelta_z(x_2, \infty ) \ge 2\}) = \nu_\alpha^+ ((x_1,\infty))= x_1^{- \alpha}$.  We shall now turn to the second sum in \eqref{eq_first_expression_lim_measure_lin_dep}. If $l \ge 1$, then we can see that $\nu_\alpha^+ (\{ z > 0 : (\bdelta_z + \bdelta_{\phi z}) (x_2, \infty) \ge 2 \mbox{ and } (\bdelta_z + \bdelta_{\phi z}) (x_1, \infty) \ge 1\})  = \nu_\alpha^+ \big( (\max(x_2, \phi^\inv x_1), \infty) \big) = [\max(x_2, \phi^\inv x_1)]^{- \alpha}$.  A simple algebra combined with the fact $\phi >1$ leads us to the following expression for the $l$-th term in the second sum in \eqref{eq_first_expression_lim_measure_lin_dep}
\aln{
& \nu_\alpha^+ \big( (\phi^\inv x_1, \infty) \big) = \phi^\alpha x_1^{- \alpha}.
}
The final expression in \eqref{eq_two_largest_order_stat_lin_dep} follows from these observations if \eqref{eq_technical_cond_lin_dep} holds.

We now provide arguments for conditions \eqref{eq_technical_cond_lin_dep} to hold. Note that ${\sf E}_{x_1, x_2} \subset {\sf E}_{x_1} = \{ \xi \in \scrm_0 : \xi(x_1, \infty) \ge 1\}$ and therefore, ${\rm cl}({\sf E}_{x_1, x_2}) \subset {\rm cl}({\sf E}_{x_1})$. We have already seen that $\varnothing \notin {\rm cl}({\sf E}_{x_1})$ and hence the first claim in \eqref{eq_technical_cond_lin_dep} follows. We now observe that ${\rm cl}({\sf E}_{x_1, x_2}) \subset \ol{\sf E}_{x_1, x_2}: = \{\xi \in \scrm_0 : \xi([x_1, \infty)) \ge 1 \mbox{ and } \xi([x_2, \infty)) \ge 2\}$ which yields the following upper bound
\alns{
m_\phi^* ( \partial {\sf E}_{x_1, x_2}) \le m_\phi^*( \ol{\sf E}_{x_1, x_2}) - m_\phi^* ({\sf E}_{x_1, x_2}).
}
We can now use the absolute continuity of the measure $\nu_\alpha^+$ with respect to the Lebesgue measure to show that the upper bound equals $0$ and the second claim in \eqref{eq_technical_cond_lin_dep} follows. Hence, the proof of \eqref{eq_two_largest_order_stat_lin_dep} is complete.

Derivation of \eqref{eq_max_lin_dep} from \eqref{eq:expression:limit:measure:linearly:dependent:disp} is very similar to the proof of Theorem~\ref{thm:ldp:rightmost:joint:regvar} and so the details are skipped.

\subsubsection{\bf Proof of Theorem~\ref{thm:extremal:process:ldp:iid:disp}} \label{proof:pp:consq:pp:iid}

Recall from the Remark~\ref{thm:ldp:pp:iid:disp:leaf} that the measure $\lambda$ equals  
\aln{
\lambda_{iid}(\cdot ) = \sum_{i =1}^\infty \bdelta_0 \otimes \ldots \otimes \bdelta_0 \otimes \underbrace{\nu_\alpha}_{i\mbox{-th position}} \otimes \bdelta_0 \otimes \ldots \otimes \bdelta_0
}
when the displacements are asymptotically tail-independent or i.i.d. So if we substitute $\lambda $ by $\lambda_{iid}$ in \eqref{eq:derived_lim_measure_K}, then we have the following expression 
\aln{
& \frac{\prob(U > 0)}{1 - p_e} \sum_{l =0}^\infty \mu^{- (l + 1)} \exptn \Big[ \sum_{{\sf G} \in \pow([1~:~\wt{U}])} \sum_{i =1}^\infty \bigotimes_{j =1}^{i -1} \bdelta_0 \otimes \nu_\alpha \bigotimes_{j'= i +1}^\infty \bdelta_0 (\{ {\bf x} \in \mbbr^\mbbn: \sum_{s \in {\sf G}} \wt{Z}_l^{(s)} \bdelta_{x_s} \in {\sf A}\})  \nonumber \\
& \h \h \h [\prob(Z_l \ge 1)]^{|{\sf G}|} [\prob(Z_l = 0)]^{\wt{U} - |{\sf G}|} \Big] \nonumber \\
& = (1 - p_e)^\inv \prob(U \ge 1) \sum_{l =0}^\infty \mu^{-(l + 1)} \exptn \Big[ \nu_\alpha \big( \{x \in \mbbr: \wt{Z}_l \bdelta_x \in {\sf A}\} \Big) \Big]  \nn \\
& \h \h \h \exptn \Big[ \sum_{{\sf G} \in \pow([1~:~\wt{U}] \setminus \{\emptyset\})} |{\sf G}| [\prob(Z_l \ge 1)]^{|{\sf G}|} [\prob(Z_l = 0)]^{\wt{U} - |{\sf G}|}  \Big]. \label{eq_cond_exptn_lim_measure_iid}
}
To derive the last equality, we have used the fact that $(\wt{Z}_l^{(i)} : i \ge 1)$ are independent copies of $\wt{Z}_l$ and also independent of $\wt{U}$ for every $l \ge 1$. We can see that the last expectation in \eqref{eq_cond_exptn_lim_measure_iid} equals 
\aln{
\exptn \Big[ \sum_{k =1}^{\wt{U}} k \binom{\wt{U}}{k} (\prob[Z_l \ge 1])^k (\prob[Z_l = 0])^{\wt{U} - k} \Big] = \exptn(\wt{U}) \prob(Z_l \ge 1) = \mu \prob(Z_l \ge 1)/ \prob(U \ge 1).
}
This observation provides the final expression of $\wt{m}^*_{iid}$ in \eqref{eq_ldp_pp_lim_iid_with_leaf}. Additionally, if the conditions of Theorem~\ref{thm:extremal:process:ldp:iid:disp} hold (the genealogical tree does not have any leaf), we use $\prob(Z_l \ge 1) = 1$ in \eqref{eq_ldp_pp_lim_iid_with_leaf} to get $m^*_{iid}$. Hence, Theorem~\ref{thm:extremal:process:ldp:iid:disp} follows. 


\subsection{Proof of Theorem~\ref{thm:ldp:maxima:iid}}

Recall that ${\sf E}_x = \{\xi \in \scrm(\mbbr_0) : \xi(x, \infty) \ge 1\}$ and $\varnothing \notin {\rm cl}({\sf E})$ from the proof of Theorem~\ref{thm:ldp:rightmost:joint:regvar}. If the displacements are asymptotically tail-independent or i.i.d., then we can use \eqref{eq_ldp_pp_lim_iid_with_leaf} to conclude
\aln{
\lim_{n \to \infty} r_n \prob^*(M_n^{(1)} > \gamma_n x) & = \lim_{n \to \infty} r_n \prob^* \big( {\bf N}_n \in {\sf E}_x \big)  = \wt{m}^*_{iid}({\sf E}_x). \label{eq_iid_max_basic}
} 
Some algebra and homogeneity property of $\nu_\alpha$ added with \eqref{eq_iid_max_basic}, yield the expression for the maximum position in \eqref{eq_ldp_max_iid_with_leaf}. The derivation would be complete if we show that $\wt{m}^*_{iid}(\partial{\sf E}_x) = 0$. The proof of the last claim is very similar to the proof of $m^*(\partial{\sf E}_x) = 0$ (see proof of Theorem~\ref{thm:ldp:rightmost:joint:regvar}) and so omitted here.

We can use $\wt{\sf E}_x = \{\xi \in \scrm_0 : \xi(- \infty, - x) \ge 1\}$ instead of ${\sf E}_x$ to derive the large deviations of the minimum position $\wt{M}_n^{(1)}$.  The expression for the minimum position  in \eqref{eq_ldp_max_iid_with_leaf} follows.

If we further assume that the genealogical tree does not have any leaf, then Theorem~\ref{thm:ldp:maxima:iid} follows by substituting $\prob(Z_l \ge 1) = 1$ and $p_e = 0$ in the expression \eqref{eq_ldp_pp_lim_iid_with_leaf}.

\section{List of Notations} \label{sec:notation}
\makeatletter{}
\label{pg:notation} To ease the reading, we list the important notions and notations used in this paper, and the corresponding page numbers. 

{\footnotesize
\begin{center} \renewcommand{\arraystretch}{1.2}
\begin{longtable}{p{2.4cm}p{10.1cm}p{1.5cm}}
  \textbf{Notation}&\textbf{Description}&\textbf{Page} \\
 $\bdelta_x$ & Dirac's delta measure which puts unit mass at $x$  & \pageref{nota_dirac_delta}\\
 ${\sfT }$ & the genealogical tree & \pageref{nota_tree_vertices_edges}\\ 
 ${\sf u}$ & a typical vertex of ${\sf T}$ & \pageref{nota_tree_vertices_edges} \\
 $|{\sf u}|$ & generation of the vertex $|\sfu|$ & \pageref{nota_tree_vertices_edges}\\
 $X(\sfu)$ & displacement attached to the vertex $\sfu$ &  \pageref{nota_tree_vertices_edges}\\
 $S(\sfv)$ & position of the vertex $\sfv$ on the real line & \pageref{nota_tree_vertices_edges} \\
$Z_i = |{\sf D}_i|$ & size of the $i$-th generation of ${\sf T}$ & \pageref{nota_tree_vertices_edges} \\
$M_n^{(k)}$ & the $k$-th largest position at the $n$-th generation & \pageref{nota_tree_vertices_edges}\\
$\mu$ & expected size of the first generation  & \pageref{eq:assumption:branching:process}\\
${\sf D}_i$ & collection of all vertices at the $i$-th generation of ${\sf T}$  & \pageref{nota_Dn} \\ 
$\mbbr_0$ & $\mbbr \setminus \{0\}$  & \pageref{nota_zero_removed_real_line} \\
$\mathscr{M}(\mbbr_0)$ & space of all point measures on $\mbbr_0$ &  \pageref{nota_zero_removed_real_line}\\
$\varnothing$ & null measure on $\mbbr_0$ & \pageref{nota_zero_removed_real_line} \\
$\scrm_0$ & $\scrm(\mbbr_0) \setminus \{ \varnothing\}$ &  \pageref{nota_zero_removed_real_line}\\
${\sf E}_x$ & $ \{ \xi \in \scrm(\mbbr_0) : \xi(x, \infty) \ge 1 \}$ & \pageref{nota_E_x} \\
${\rm int}(\sf A)$ & interior of the set ${\sf A}$ & \pageref{nota_E_x} \\
${\rm cl}({\sf A})$ & closure of the set ${\sf A}$ & \pageref{nota_E_x}\\
$\partial {\sf A}$ & boundary of the set ${\sf A}$ &  \pageref{nota_boundary_set}\\
$\mathbb{M}(\scrm_0)$ & space of all measures on $\scrm_0$ & \pageref{nota_measure_on_Mzero}\\
$\mathbb{M}_0$ & a topology on the space of all measures on the punctured Polish space   & \pageref{not_Mzero_topology} \\ 
$F_{g_1, g_2, \epsilon_1, \epsilon_2}$ & convergence determining class of functionals for the $\mathbb{M}_0$ convergence on the space $\scrm_0$ & \pageref{eq_defn_hlsfun} \\
$\mathbb{R}^\mbbn_{\bf 0}$ & the space of all real sequences except the origin ${\bf 0}_\infty$ & \pageref{not_punctured_space_real_sequence} \\
$\widetilde{Z}_l$ & the random variable $Z_l$ conditioned to stay positive & \pageref{nota_Zl_condition_to_stay_positive}\\
 $\widetilde{U}$ & the random variable $U$ conditioned to stay positive & \pageref{nota_u_condition_to_stay_positive}\\
$\pi_{j_1, j_2, \ldots, j_k}$ & $\pi_{j_1, j_2, \ldots, j_k} : \mbbr^\infty \to \mbbr^k$ such that $\pi_{j_1, j_2, \ldots, j_k} ({\bf x}) = (x_{j_1}, x_{j_2}, \ldots, x_{j_k})$ & \pageref{nota_projection_map_general}\\ 
$|{\sf G}|$ & cardinality of the set ${\sf G}$  & \pageref{nota_cardinality_set}\\
$\pow({\sf G})$ & power set of ${\sf G}$ & \pageref{nota_cardinality_set} \\
$\emptyset$ & null set & \pageref{nota_cardinality_set}\\
$[i~:~j ]$ & $\{i, i + 1, i + 2, \ldots, j\}$ & \pageref{nota_closed_interval_integers} \\
$p_e$ & probability of extinction of a Galton-Watson tree with progeny distribution $Z_1$ & \pageref{nota_pe} \\
${\cal S}$ & survival event for the random tree ${\sf T}$ & \pageref{nota_pe}\\${\varpi}$ & root of the genealogical tree ${\sf T}$ & \pageref{nota_Iv} \\
 ${\sf I}(\sfv)$ & collection of vertices on the genealogical path to the vertex $\sfv$ from the root of ${\sf T}$ & \pageref{nota_Iv}  \\
 $\sfu \rightarrow \sfv$ & geodesic path from the vertex $\sfu$ to $\sfv$ & \pageref{nota_geodesic_path} \\
$A(\sfu)$ & the number of descendants of the vertex $u$ in the $n$-th generation of ${\sf T}$ & \pageref{nota_Au}\\
$Z(\sfu)$ & the number of children of the vertex $\sfu$ &  \pageref{nota_Zu}\\
$\prob^*$ & conditional probability given the survival of the genealogical tree & \pageref{nota_pe} \\
$\exptn^*$ & conditional expectation induced by $\prob^*$ & \pageref{nota_pe} \\
$\mathbb{T}_\kappa^{(1)}$ & pruned subtree with $\kappa$-generations & \pageref{not_pruned_subtree}\\
$\mathscr{T}_\kappa^{(1)}$ & regularized version of the subtree $\mathbb{T}_\kappa^{(1)}$ & \pageref{not_regular_tree}\\
${\rm PROJ}_\varrho$ & projection map ${\rm PROJ}_\varrho = \pi_{1, 2, \ldots, \varrho} : \mbbr^\mbbn \to \mbbr^\varrho$ for every $\varrho \ge 1$ & \pageref{not_proj}\\
 \end{longtable}
\end{center}
}

\section*{Acknowledgment}
The research is supported by the NWO VICI grant 639.033.413 and partially supported by Polish National Science Centre under the grant 2018/29/B/ST1/00756 and SEED Grant (RD/0520-IRCCSH0-009) provided by IRCC, IIT Bombay.  The author is thankful to Rajat Subhra Hazra, Zbigniew Palmowski, Parthanil Roy, Gennady Samorodnitsky and Bert Zwart for many helpful discussions and also would like to acknowledge the warm hospitality of ISI, Bangalore for the period  15 - 26th March, 2017 where the project started. A part of the research was done when the author was a visiting scientist in ISI, Kolkata for the period May 22 - August 31, 2017.  \rtext{The author is thankful to the anonymous referees for their comments and suggestions which improved the readability of the paper and accuracy of the results.}

\bibliographystyle{abbrvnat}

\vspace{1.5cm}

\parbox[t]{12cm}{\footnotesize {Ayan Bhattacharya \\ Department of Mathematics,\\ Indian Institute of Technology Bombay,\\ Mumbai, Maharashtra, India.\\eamil - ayanbh@math.iitb.ac.in}}

\end{document}